\documentclass[11pt, reqno]{amsart}
\usepackage{amsmath, amsthm, amscd, amsfonts, amssymb, graphicx, color, tikz, mathrsfs}

\usepackage[bookmarksnumbered, colorlinks, plainpages]{hyperref}
\usetikzlibrary{decorations.pathreplacing,calligraphy}
\newtheorem{theorem}{Theorem}[section]
\newtheorem{lemma}[theorem]{Lemma}
\newtheorem{proposition}[theorem]{Proposition}
\newtheorem{corollary}[theorem]{Corollary}
\theoremstyle{definition}
\newtheorem{definition}[theorem]{Definition}

\theoremstyle{remark}
\newtheorem{remark}[theorem]{Remark}
\numberwithin{equation}{section}

\usepackage[authoryear]{natbib} 
\usepackage{hyperref}
\hypersetup{
    colorlinks=true,
    linkcolor=blue,
    citecolor=red,
    urlcolor=cyan
}
\begin{document}
\setcounter{page}{1}
\setcounter{figure}{0}

\centerline{}

\centerline{}

\title[Mixing Flows on Finite Area Translation Surfaces]{On Mixing Flows on Finite Area Translation Surfaces}

\author[Erick Gordillo]{Erick Gordillo}

\address{$^{1}$ Department of Mathematics, University of Heidelberg}
\email{\textcolor[rgb]{0.00,0.00,0.84}{egordillo@mathi.uni-heidelberg.de}}

\begin{abstract}

We construct an explicit family of finite-area, infinite-genus translation surfaces whose vertical translation flow is strongly mixing. This provides a positive answer to a question posed by Lindsey and Treviño~\cite{LT}..
\end{abstract} \maketitle
\section{Introduction}

Translation surfaces are geometric objects in which it is possible to define a notion of geodesics for every direction $\theta\in [0,2\pi)$. The collection of geodesics for a direction $\theta$ yields  a \textit{foliation}\footnote{It is not a foliation in the strict sense since there are geodesics whose domain of definition can not be extended to $(-\infty,\infty)$. Nevertheless, in most of the cases, in particular in the ones that we are going to consider the set  of these \textit{problematic} geodesics has measure  zero.}. In the last decades people have been interested in the distribution of the leaves of this foliation, or in the dynamical properties of the flow (usually called translation flow or billiard flow) that define this foliation, from an ergodic theory point of view. \\

The study of translation flows on surfaces of finite type—that is, surfaces homeomorphic to compact Riemann surfaces—has yielded significant results. In particular, in \cite{KMS}, the authors establish the unique ergodicity of the translation flows, showing that for almost every direction, the leaves of the associated foliation are equidistributed. Furthermore, in \cite{AF}, it is proven that for almost every translation surface in a given stratum, the translation flow is weakly mixing for almost every direction.\\

In recent years, there has been growing interest in understanding the geometry and dynamics of translation surfaces beyond the compact setting. As expected, many properties established for compact translation surfaces no longer hold, while new phenomena emerge. A comprehensive overview of recent developments can be found in \cite{DHV}.  \\

In this work, we study one such property and establish the following result:

\begin{theorem}\label{teointro}
Let $S$ be a finite-area, infinite-genus translation surface constructed from an $L$-shaped polygon satisfying the following conditions:

\begin{enumerate}
    \item The vertical heights $p$ and $q$ of the $L$-shaped polygon are rationally independent, with $p > q$. Let $x_0 \in (0,1)$ denote the point separating the base interval into two parts: $[0,x_0)$, corresponding to the non-spacer region, and $[x_0,1)$, corresponding to the spacer region of a mixing staircase transformation $T$. The roof function over $[0,x_0)$ has height $p$, while over $[x_0,1)$ it has height $q$.
    
    \item The verticall edges are glued by identifying opposite sides.
    
    \item The horizontal edges are glued according to a mixing staircase transformation $T$ defined on the base interval $[0,1)$. Specifically, a point $(x,0)$ is identified with $(T(x), \epsilon(T(x)))$, where $\epsilon(T(x)) \in \{p,q\}$ encodes the height of the corresponding roof.
\end{enumerate}

Then the vertical translation flow $\phi_t$ on $S$ is mixing.
\end{theorem}

This result serves as an example highlighting the contrasting properties of finite and infinite translation surfaces. The contrast arises from a key contribution by  Katok in the 1980s \cite{KA}, where he established that suspension flows constructed over interval exchange transformations under functions of bounded variation are never mixing. This consequently implies that translation flows on surfaces of finite type can never be mixing. In 2007, Ulcigrai demonstrated that for almost any interval exchange transformation, suspension flows constructed under functions with certain singularity properties exhibit mixing behavior \cite{Ulcigrai}.\\

Nevertheless, although it is true that these flows are volume preserving, they are not translation flows. \\

Infinite translation surfaces can have infinite area, with interesting examples such as the wind tree model \cite{AH} or $\mathbb{Z}$-coverings \cite{HW}, as well as finite area surfaces, as shown in \cite{Tre1}. In \cite{LT}, Lindsey and Treviño use techniques from cutting and stacking constructions to study the dynamics of translation flows on finite area translation surfaces, encoding their information in bi-infinite Vershik--Brattelli diagrams. One particularly interesting result they present is an extension of Rudolph's classical representation theorem for aperiodic flows \cite{R}. Specifically, they prove in \mbox{\cite[Proposition 7.6, Lemma 7.7]{LT} } that any aperiodic flow on a finite Lebesgue measure space with finite entropy is measure-theoretically isomorphic to the vertical flow of a translation surface of finite area arising from a Vershik--Brattelli diagram.\\

Importantly, this result implies the existence of mixing translation flows. For instance, Anosov flows provide excellent examples of mixing flows that satisfy the conditions outlined in their result. Consequently, for an Anosov flow $\gamma_t$, there exists a translation surface of finite area $S$ whose vertical flow $\phi_t$ is measure-theoretically equivalent to $\gamma_t$; in particular, this vertical translation flow is mixing.\\

In their work, K. Lindsey and R. Treviño do not provide specific examples of such flows. However, they conjecture \cite[Conjecture 7.8]{LT} that suspensions over mixing staircase transformations can yield mixing translation flows. Theorem \ref{teointro} provides a positive answer to this question.\\

\section{Organization}

Our goal is to construct a family of examples of infinite translation surfaces with finite area such that their translation flow is mixing.\\

In Section \ref{preliminaries} we provide  basic definitions of translation surfaces, rank-one transformations and staircase transformations, suspension flows and a way to construct a translation surface from a cutting and stacking process. \\

In Section \ref{SectionClassical}  we prove the result on a suspension flow over the classical staircase. \\

\begin{theorem}\label{teo2.1}
    Let $T$ be the classical staircase transformation and $p,q$ two rationally independent positive numbers. Let $f$ be a 2-wise constant function such that $f$ attains the values $p,q$ on complementary intervals. Then the suspension flow $([0,1],T,f,\mu)$ is  mixing.
    \end{theorem}

 The motivation to start with this case is due to the fact that the classical staircase $T$ was the first staircase transformation which was proved to be mixing \cite{Adams}.\\

The strategy to establish the mixing property for suspension flows over any mixing staircase transformation follows essentially the same approach as in the classical staircase case. Accordingly, in Section \ref{SectionClassical} we present the full details of the argument. In Section \ref{gen}, we focus only on the modifications required for the generalized case and explain why these adjustments are valid. This allows us to conclude:

\begin{theorem}\label{teo2.2}
    Let $T$ be any mixing staircase transformation and $p,q$ two rationally independent positive numbers. Let $f$ be a 2-wise constant function such that $f$ attains the values $p,q$ on complementary intervals. Then the suspension flow $([0,1],T,f,\mu)$ is  mixing.
    \end{theorem}

Note that the graph of the function $f$ from Theorems \ref{teo2.1} and \ref{teo2.2}, together with the $x$-axis and $y$-axis, forms an $L$-shaped polygon. By identifying the bottom edges with the top edges using the staircase transformation $T$ and the vertical edges with their opposite edges, we obtain a translation surface with finite area and infinite genus. Moreover, the vertical translation flow on this surface corresponds to the suspension flows described in both theorems. Therefore, Theorems \ref{teo2.1} and \ref{teo2.2} together yield Theorem \ref{teointro}.\\

We believe that it should be possible to impose some \textit{mild} conditions on the growth of spacers and cutting sequences such that these results could be extended to more general rank-one transformations satisfying these conditions. For example, the restricted growth conditions on the cutting sequences and additional conditions on the spacer sequences, as explained in \cite{CS1}, \cite{CS2}.\\

It is almost immediate from the construction of the suspension flows considered in Theorem \ref{teo2.1} and Theorem \ref{teo2.2} that they can be regarded as \textit{rank-one flows}. In the 1990s, Ryzhikov proved that if rank-one flows are mixing, then they are also $n$-fold mixing for any $n$ \cite{Ryz}. Therefore, we immediately obtain this result in Section \ref{nfold}:

\begin{theorem}\label{teo2.3}
    Let $T$ be  any mixing staircase transformation and $p,q$ two rationally independent positive numbers. Let $f$ be a 2-wise constant function such that $f$ attains the values $p,q$ on complementary intervals. Then the suspension flow $([0,1],T,f,\mu)$ is  mixing of any order.
    \end{theorem}
    
\section{Preliminaries}\label{preliminaries}
\subsection{Translation Surfaces}
There are various approaches to defining a translation surface, each tailored to different interests.\\

One common depiction of a translation surface involves considering a connected topological surface $S$ alongside a discrete subset $ \Sigma \subseteq S$. Within $ S \setminus \Sigma$, a maximal atlas called the \textit{translation atlas} $ \{ \phi_\alpha, U_\alpha \}_{\alpha \in \Lambda}$ is defined, where the transition maps are Euclidean translations.\\

For the purposes of studying billiard dynamics within polygons and generalized polygons, a more practical definition of a translation surface involves an at-most-countable collection of convex polygons $\{ P_n \} \subset \mathbb{R}^2 $, where each pair of edges is identified by translations. Denoting the collection of vertices of these polygons as  $ V $, we define the surface as
$$ S^\circ = \left( \left( \bigcup_{n \in \mathbb{N}} P_n \right) \setminus V \right) / \sim. $$
If $S^\circ$ is a connected space, we call $S$, the metric completion of $S^\circ$, a translation surface.
\\

These two definitions are equivalent (see \cite{DHV}). A translation surface $S$ is categorized as finite type if the collection of polygons is finite, or if the topological surface $S$ on which the translation atlas is defined is homeomorphic to a compact surface $X$; otherwise, it is termed infinite type.
\\

Now, for any direction $[0,2\pi)$, we have a well-defined translation flow in $\mathbb{C} = \mathbb{R}^2$ given by
$$ F_{\theta}^t(z) = z + t e^{i\theta}, $$
which yields a constant vector field
$$ X_{\theta} = \frac{\partial F_{\theta}^t}{\partial t} \bigg|_{t=0}(z). $$

The definition of a translation surface allows us to pull back this vector field to $S \setminus \Sigma$. For any $z \in S \setminus \Sigma$, we can consider a maximal integral curve $\gamma_z: I \to S \setminus \Sigma$ such that $\gamma_z(0) = z$, and we can define a flow on $S \setminus \Sigma$ as
$$ \phi_{\theta}^t(z) = \gamma_z(t). $$

This is the \textit{translation flow} in direction $\theta$. When the direction is not stated, meaning only writing $\phi_t$ we will mean the translation flow in the vertical direction.

\subsection{Ergodic Theory}

\begin{definition}

If $(X,T,\mathcal{B},\mu)$ is a measurable dynamical system, we say $T$ is an ergodic transformation if for every measurable invariant set $A$  such that $T^{-1}(A)=A$ implies  $\mu(A)\in\{0,1\}$.
\end{definition}

In terms of the Birkhoff theorem it can be stated as for every measurable function $f$  $$\lim \frac{1}{n}\sum_{i=0}^n f\circ T^n(x)\to \int_X fd\mu, $$
 for  $\mu-$almost every $x$.\\
 
In the ergodic theory category we have the following chain of implications:

$$ Mixing \to Weakly\:\: Mixing \to Ergodic.$$
 
\begin{definition}
    
We say that a transformation $T$ is  mixing if for every $A,B\in \mathcal{B}$ $$\lim_{n\to \infty}\mu(T^{-n}(A)\cap B)= \mu(A)\mu(B).$$

\end{definition}

In contrast to weakly mixing where the convergence of the splitting of measures is a Cesaro-convergence. Meaning $T$ is weakly mixing when for every measurable sets $A,B$ 
$$\lim_{n\to\infty} \frac{1}{n}\sum_{i=0}^{n-1}|\mu(T^{-n}(A)\cap B)-\mu(A)\mu(B)|=0.$$

Our aim is to study the continuous version of this properties.

\begin{definition}

A flow $\phi_t$ is a family of measurable transformations $\{\phi_t\}_{t\in \mathbb{R}}$ on a measurable space $(X,\mathcal{B},\mu)$ parametrised with $\mathbb{R}$ such that $\phi_{t+s}=\phi_t(\phi_s).$
\end{definition}

\begin{definition}
We say that the flow $\phi_t$ on a measurable space $(X,\mathcal{B},\mu)$ is  mixing  if for every $A,B\in\mathcal{B}$ $$\lim_{t\to\infty} \mu(\phi_{-t}(A)\cap B)= \mu(A)\mu(B).$$

As we shall see in following sections, the notion of equidistribution plays an important role.\\
\begin{definition}
    A sequence of numbers $\{x_n\}$ is \textit{equidistributed} in the interval $[a,b]$ if for every subinterval $[c,d]$ then $$\lim_{n\to\infty}\frac{ s_n(c,d)}{n}=\frac{c-d}{b-a}.$$
    Where $s_n(c,d)$ is the cardinality of the  set of numbers $1\leq k\leq n$ such that $x_k\in [c,d]$.
\end{definition}

\begin{remark}
    By the equidistribution theorem, if $q$ is an irrational number then the sequence $\{nq \mod{1}\:\: n\in \mathbb{N}\} $ is equidistributed in the unitary interval.
\end{remark}

 \begin{definition}
     A transformation $T:X\to X$ is said to be uniquely ergodic if there exists only one invariant measure for $T$. 
 \end{definition}

 When a transformation is uniquely ergodic, the unique invariant measure is ergodic, and the conclusion of the Birkhoff ergodic theorem applies for \textit{every} $x\in X$. \\

 Observe that if $T$ is a uniquely ergodic transformation in an interval $[a,b)$ then for any $x$, the sequence $x_n=T^n(x)$ fulfills that for every measurable function $f$ 
 $$\lim_{n\to \infty} \frac{1}{n}\sum_{i=0}^n f(x_n)\to\frac{1}{b-a} \int_a^b f(x)dx,$$
in particular the sequence $x_n$ equidistributes in the interval $[a,b)$.

 \end{definition}
\subsection{Cutting and stacking constructions}
Cutting and stacking constructions represent a powerful tool extensively utilized in ergodic theory to construct measurable invertible transformations on the interval. Intuitively, they can be conceptualized as a limit process involving compositions of periodic maps on subintervals. This process entails considering a sequence of transformations $T_1, T_2, \ldots, T_n, \ldots$ satisfying the following properties:

\begin{enumerate}
    \item \textbf{Domain}$(T_i)\subseteq \textbf{Domain}(T_{i+1})$.

    \item $T_{n+1|_{\textbf{Domain}(T_n)}}=T_n.$
\end{enumerate} 
We can define a map $T$ such that $$\textbf{Domain}(T)=\bigcup_{n\in \mathbb{N}}\textbf{Domain}(T_n).$$
And such that $T$ is the pointwise limit of the sequence $T_n$.\\

We are going to pay attention to a class of transformations derived from cutting and stacking constructions, referred to as \textit{rank-one transformations}. These transformations originate from a geometric framework inspired by Rokhlin's column theorem \cite{Cornfeld}. Their significance was highlighted by Ornstein \cite{Ornstein}, who employed them to construct examples of interval transformations demonstrating mixing properties without the presence of a square root. Ornstein's stochastic argument established the almost sure mixing property of these transformations. Subsequent contributions by Adams \cite{Adams},  or Creutz and Silva \cite{CS1}, \cite{CS2} and others have provided concrete examples and conditions under which these transformations exhibit mixing behavior.

\subsubsection{Rank-One Transformations}

We provide an algorithmic definition of such transformations. Consider an interval $I=[0,b)$ with $b$ possibly infinite, and a subinterval $J=[0,a)$ with $a$ necessarily finite. We denote the length of the interval $J$ as $A$. The interval $J$ serves as the initial step of the transformation, with the initial transformation $T_0$ being the identity. Interval $I$ acts as the domain of the transformation $T$. \\

Now, let $\{r_n\}$ be a sequence of positive integers. Firstly, we divide interval $J$ into $r_1$ subintervals $J_i$ with equal lengths $\frac{A}{r_1}$. Then, we introduce another sequence of non-negative integers $\{s_n\}$. Each interval $[a,a+\frac{s_1A}{r_1})$ is divided into $s_1$ equal parts, and each of these subintervals, referred to as spacers, is placed above one of the parts into which we divided the interval $J$.\\ 

After adding all these spacers above the corresponding subintervals, the set of all intervals stacked above each subinterval $J_i$ of the partition of $J$ is called a column $C_{1,i}$. We stack all these columns one above the other, such that column $C_{1,i}$ lies beneath $C_{1,i+1}$ for each $i\leq r_1-1$. The height $h_1$ of this first step of the transformation $T$  is given by $$h_1=r_1+ \sum_{i=1}^{r_1}s_{1,i},$$ where $s_{1,i}$ is the number of spacers added above interval $J_i$. \\

We define a transformation $T_1$ on the first $h_1-1$ intervals of this column $\cup_i C_{1,i}$ such that $$T_1(x,j)=(x,j+1),$$ where the second entry indicates the interval in the column. Throughout this text, this column will be referred to as the first step of the configuration of the transformation $T$.\\

To define the transformation $T_2$, we partition the first step of the transformation's configuration into $r_2$ equal parts, resulting in $r_2$ new subcolumns. We further divide the interval $[a+\frac{As_1}{r_1},a+\frac{As_1}{r_1}+ \frac{As_2}{r_1r_2})$ into $s_2$ equal parts, each denoted as spacers. These new spacers are positioned above  the $r_2$ new subcolumns. The number of spacers above each subcolumn is denoted as $s_{2,i}$. Similar to the previous step, we stack the new subcolumns $C_{2,i}$ one above the other. \\

The height $h_2$ is given by $$h_2=r_2h_1+ \sum_{i=1}^{r_2}s_{2,i}.$$ The transformation $T_2$ is defined on the first $h_2-1$ intervals as $T_2(x,j)=(x,j+1)$. It's important to note that $T_2$ coincides with $T_1$ in the intervals corresponding to the first step of the configuration. This new stack is referred to as the second step of the configuration of the transformation $T$.\\

Continuing this process, suppose we have defined the $n$th step of the configuration of the transformation $T$, with height $h_n$ and a well-defined \mbox{transformation $T_n$}. To define the next step, we partition the column with $h_n$ intervals into $r_{n+1}$ equal parts, resulting in $r_{n+1}$ new subcolumns. We further divide the interval $$[a+\sum_{i=1}^{n}\frac{s_nA}{\prod_{j=1}^{i}r_j},a+\sum_{i=1}^{n+1}\frac{s_{n+1}A}{\prod_{j=1}^{i}r_j}) $$ into $s_{n+1}$ equal parts, and they are positioned above each of the $r_{n+1}$ new subcolumns according to the provided extra information. The number of intervals attached to the top of each subcolumn is denoted as $s_{n+1,i}$. 
\\

We stack each of these new stacks one above the other, similar to previous cases. The height is given by $$h_{n+1}=r_{n+1}h_n+\sum_{i=1}^{r_{n+1}}s_{n+1,i}.$$ The transformation $T_{n+1}$ is defined such that $T_{n+1}(x,j)=(x,j+1)$ for each \mbox{$j\leq h_{n+1}-1$.}

Continuing this process to infinity, we obtain a transformation $T$ defined on the interval $$[0,b)=[0, \lim_{n\to\infty}\frac{h_n A}{r_1....r_n}).$$
\begin{definition}
    The sequences $\{r_n\}$ and $\{s_n\}$ are referred as \textit{cutting sequence} and \textit{spacer sequence} respectively, for the transformation $T$ that comes from the described cutting and stacking construction.
\end{definition}
\begin{definition}
    We call the interval $J=[0,a)$ the non-spacer interval and the interval $I\setminus{J}$ as the spacer interval.
\end{definition}

Under some  conditions (for example, when $\frac{r_n^2}{h_n}\to 0$), this interval $[0,b)$ is finite, which is the case we are interested in. Such conditions are referred to as \textit{restricted growth}. In \cite{CS1}, Creutz and Silva prove that rank-one transformations with restricted growth whose stacking spacer sequence $\{s_n\}$ is uniformly ergodic respect $T$  are  mixing. Later in \cite{CS2} the authors proved that a rank one transformation $T$ is mixing if and only if the spacer sequence $\{s_n\}$ is slice ergodic respect $T$. Additionally, by a result of Ryzhikov \cite{Ryz}, these transformations are $n$-fold mixing for any $n$. See Section \ref{nfold} for more details.

\subsubsection{Staircase Transformations}
We are going to define a particular class of rank-one transformations which are relevant for this paper.
\begin{definition}
    Let $T$ be a rank-one transformation, and $\{r_n\}$ its cutting sequence. We say $T$ is a \emph{Staircase Transformation} if for any step $n$ of the configuration, the number of spacers $s_{n,i}$ added  to the subcolumn $C_{n,i}$ is $i-1$ for $1\leq i\leq r_n$ and after adding the spacers, the subcolumn $C_{n,i+1}$ is stacked over the subcolumn $C_{n,i}$ for every $1\leq i \leq r_n-1$.
\end{definition}

The most typical example of a staircase transformation is the one regarded as the \emph{classical staircase} transformation, in which the cutting sequence is given by $r_n=n$. This was the first example of a staircase transformation to be proved to be  mixing \cite{Adams}.

\subsection{Translation Surfaces from Rank-One Transformations}

The description we present here is a specific instance of a broader framework outlined in \cite{LT}.

Let's recall the definition of a flow constructed under a function, often referred to as a suspension flow.\\

Let $ T:I\to I$ be a measurable transformation on the unit interval, and let \mbox{$f:I\to \mathbb{R}^+ $} be a bounded, square-integrable function that is strictly positive. Consider the space
$$ X=\{(x,y)|\: x\in I, \: y<f(x)\}. $$

We define a suspension flow $\phi_t=(I,T,f,\mu)$ as follows: for $(x,y)\in X$, the flow $\phi_t(x,y)=(x,y+t)$ if $y+t<f(x)$, and $\phi_t(x,y)=(T(x),0)$ if $y+t=f(x)$. This flow preserves the Lebesgue measure $\mu$ and provides an identification of the interval $I$ with the  graph of $f$.\\

Now, consider a transformation $T$ derived from a cutting and stacking process, and let $f:I\to \mathbb{R}^+$ be such that for some  $c\in I$ point of discontinuity for $T$, $f(x)=p$ for every $x\in [0,c)$, and $f(x)=q$ for every $x\in [c,1]$. Moreover, we require that $f$ is constant on each subinterval of continuity of the transformation $T$.\\

Observe that from this construction, we obtain $S$ together with the segments\\ $\{0\} \times [0,p]$, $\{c\}\times[ \min(\{p,q\},\max(\{p,q\})]$, and $\{1\}\times [0,q]$ as a rectangle or an $L-$shaped polygon.\\

Without loss of generality, we may assume that $p>q$. We do the following identification:

\begin{enumerate}
    \item Identify $\{0\}\times [0,q]$ with $\{1\}\times [0,q]$.\\
    
    \item Identify $\{c\}\times [q,p]$ with $\{0\}\times [q,p]$.\\
    
    \item If $f(x)=\epsilon(x)\in \{p,q\}$, identify $(x,0)$ with $(T(x),\epsilon(x))$.
\end{enumerate}

See Figure \ref{Lshaped}. Observe that after these identifications, we obtain a translation surface $S$, and the suspension flow $\phi_t$ built under $f$ is isomorphic in the measure-theoretical sense to the vertical translation flow in the surface $S$. 

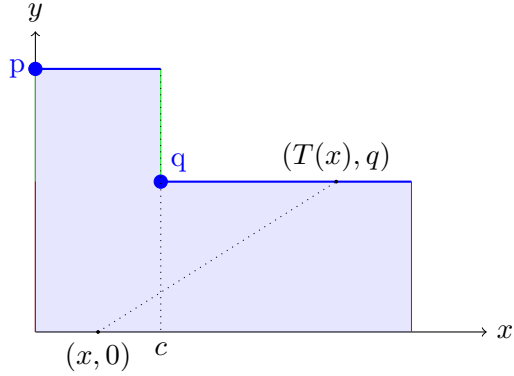
\begin{figure}
    \centering

\begin{tikzpicture}[scale=5]
    % Axis
    \draw[->] (0,0) -- (1.2,0) node[right] {$x$};
    \draw[->] (0,0) -- (0,0.8) node[above] 
    {$y$};
    \draw[red] (0,0)-- (0,.4);
    \draw[red] (1,0) -- (1,.4);
    \draw[green] (0,.4)--(0,.7);
    \draw[green] (1/3,.4) -- (1/3,.7);
    \draw[dotted] (1/6,0) -- (.8,.4);

    % Function
    \fill[blue, opacity=0.1, domain=0:1/3, variable=\x]
        (0,0) -- plot (\x,0.7) -- (1/3,0) -- cycle;
    \fill[blue, opacity=0.1, domain=1/3:1, variable=\x]
        (1/3,0) -- plot (\x,0.4) -- (1,0) -- cycle;
    \draw[blue, thick, domain=0:1/3] plot (\x,0.7);
    \draw[blue, thick, domain=1/3:1] plot (\x,0.4);
    
    \filldraw[blue] (0,0.7) circle (0.5pt) node[left] {p};
    \filldraw[blue] (1/3,0.4) circle (0.5pt) node[above right] {q};
    
    \draw[dotted] (1/3,0.7) -- (1/3,0) node[below] {$c$};
    \filldraw[black] (1/6,0) circle(.1 pt) node[below]{$(x,0)$};
     \filldraw[black] (.8,.4) circle(.1 pt) node[above]{$(T(x),q)$};
\end{tikzpicture}
   \caption{Identify by translation the red sides and green sides and $(x,0)$ with $(T(x),q).$}
    \label{Lshaped}

\end{figure}

\section{ Mixing in the case of the classical staircase} \label{SectionClassical}

To illustrate the techniques we employ, we focus on the classical staircase transformation, 
defined by the cutting sequence $r_n = n$. This example represents a staircase with restricted 
growth and was the first instance of a staircase transformation proven to be mixing 
\cite{Adams}. We begin by introducing a class of functions $f$, which will serve as the 
foundation for constructing our flows.

\begin{definition}\label{2wise}
    Let $f$ be a measurable strictly positive function defined on the unit interval $[0,1]$ characterized by constant values in two intervals with rationally independent heights $p,q$ with $p>q$. Moreover let $T$ be a the classical staircase, assume that $f_{|non-spacer}=p$ and $f_{|spacer}=q$. We call this function as \textit{ $2-$wise constant}.
\end{definition}

In this section, our aim is to establish the  mixing property of the flow $\phi_t$ constructed under a classical staircase transformation $T$ and a roof $2-wise$ constant function $f$.  To accomplish this, we will rely on a  result by Ulcigrai \cite{Ulcigrai}, which outlines conditions for  a flow constructed under a given function and a measurable interval transformation to be  mixing. Formally our result says:
\\
 
\begin{theorem}\label{teo4.1}
    Let $T$ be the classical staircase transformation and let $p$ and $q$ two positive rationally independent numbers. Let $f$ be a 2-wise constant function such that $f$ attains the values $p,q$ on complementary intervals with the additional property that $x\in [0,1]$ lies in a spacer interval if and only if $f(x)=q$. Then the suspension flow $([0,1],T,f,\mu)$ is  mixing.
    \end{theorem}

\textbf{Notation.} Throughout the text, we will denote the Lebesgue measure on $\mathbb{R}$ as $\lambda$, and the Lebesgue measure on $\mathbb{R}^2$ as $\mu$.\\

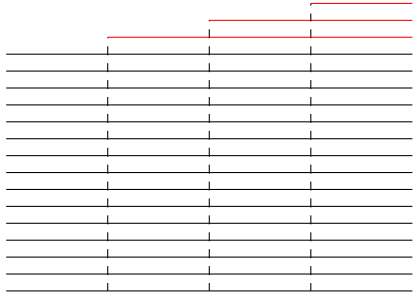
\begin{figure}
    \centering

    \label{Spacers}

\begin{tikzpicture}[scale=.45]

  \foreach \y in {0,0.5,1, 1.5,2,2.5, 3, 3.5, 4, 4.5, 5, 5.5, 6, 6.5, 7} {
   
      \draw (0,\y) -- (12,\y);
    
        \draw[dashed] (3,\y) -- (3,\y+0.5);
    \draw[dashed] (6,\y) -- (6,\y+0.5);
    \draw (9,\y)[dashed] -- (9,\y+0.5);

  }
  \draw[red] (3,7.5) -- (12,7.5);
  \draw[red] (6,8) -- (12,8);
  \draw[red] (9,8.5) -- (12,8.5);

  \draw[dashed] (6,7.5)-- (6,8);
  \draw[dashed] (9,7.5) -- (9,8.5);
 
\end{tikzpicture}
\caption{Here we consider the $4th$ step of the classical staircaise transformation, where the first $14$ levels have length equal $\frac{A}{6}$, then we add $6$ spacers of lenth $\frac{A}{24}$.}
\end{figure}

Following the notation in \cite{Ulcigrai}, we consider \emph{partial partitions}\footnote{
A partial partition refers to a finite collection of subintervals of $[0,1]$ whose union 
is a subset $A \subset [0,1]$.} $\eta(t)$ of the interval $[0,1]$, parameterized by 
$t \in \mathbb{R}_+$. Each $\eta(t)$ consists of finitely many subintervals. 
We denote by $Mesh(\eta(t))$ the maximal length of these subintervals, 
and by $\lambda(\eta(t))$ the total measure, i.e.\ the sum of the lengths of all 
subintervals in the partition.

\begin{lemma} (\cite{Ulcigrai}). \label{LemaUlci}

    Let $\phi_t$ be a suspension flow built under a measurable transformation $T$ on the interval $[0,1]$ and a roof function $f$. Assume for every rectangle $R\subseteq X$ and every $\epsilon,\delta>0$ there is a $t_0>0$ such that for every $t\geq t_0$ there exists a partial partition $\eta(t)$ of $[0,1)$ such that:
    \\
    \begin{enumerate}
        \item $\lambda(\eta(t))>1-\delta$ and $Mesh(\eta(t))\leq \delta$.
        \\
        \item For every $I\in \eta(t)$ we have that $\lambda(I\cap \phi_{-t}(R))\geq (1-\epsilon)\mu(R)\lambda(I).$
    \end{enumerate}

    Then the flow $\phi_t$ is mixing.
\end{lemma}

\subsubsection{Outline of the proof of Theorem \ref{teo4.1}}

Staircase transformations arise from the Rokhlin tower theorem. This theorem implies that for each step in the construction of the transformation, we obtain a finite number of subintervals whose lengths strictly decrease at each step. As the process continues indefinitely, the union of these subintervals will cover the interval $[0,1]$. These unions at each step provide natural candidates for partial coverings of $[0,1]$.\\

For any $\delta > 0$, there exists an $m$ such that for any subinterval $I_j^{(m)}$ from the $m$-step of the staircase transformation configuration, the following conditions are met:

\begin{enumerate}
    \item The measure of $I_j^{(m)}$ satisfies that $\lambda(I_j^{(m)})<\delta$.

    \item The measure of the union of these subintervals satisfy that
    $$\lambda\left(\bigcup_{i=1}^{h_m}I_i^{(m)}\right)>1-\delta.$$
\end{enumerate}  

Therefore for that $\delta$, the partial partition $\eta$ is the union of those intervals. Note that the partial partition depends only on $\delta$ and not $t\in\mathbb{R}_+$.\\

It is noteworthy that for any interval \( I^{(m)} \), the application of the flow \(\phi_t\) will produce a partition of \( I^{(m)} \) into several horizontal segments. Our goal is to approximate the distribution of these horizontal segments.
\\

To prove that the partitions arising from the staircase construction satisfy the second requirement of Lemma \ref{LemaUlci}, we first consider the cases where\\
\( R \subseteq [0, x_0] \times [0, p] \) and \( R \subseteq [x_0, 1] \times [0, q] \), with \( x_0 \) being the point that divides the spacer and non-spacer intervals. These two rectangles yield a natural partition of the $L$ shaped polygon, the first is the \textit{ non-spacer} part of the polygon and the second is the \textit{ spacer} part of the polygon. Then, we have to approximate the product of measures \(\mu(R) \lambda(I)\) by \(\lambda(\phi_t(I) \cap R)\).  For the remainder of the paper we assume that $(p,q)=(1,q)$.  
\\

First, we prove a lemma inspired by the fact that a staircase transformation \( T \) exhibits mixing properties \cite{Adams}. Lemma~\ref{lemma4.2} states that there exists an $m_0$ such that for every $n\geq m_0$,  any level \( I^{(n)} \) at the \( n \)-th step of the staircase transformation, and any rectangle \( R \), the measure  
\[
\lambda(\pi_x(\phi_t(I^{(n)}))\cap \pi_x(R))
\]  
converges to \( \lambda(I^{(n)})\lambda(\pi_x(R)) \).  \\

Next, we establish that for any level at any step of the staircase construction, the action of the flow \( \phi_t \) on this interval produces an equidistributed set of heights. More precisely, given any level \( I^{(n)} \), we define the \( I^{(n)} \)-cylinder as either \( I^{(n)} \times [0,1) \) or \( I^{(n)} \times [0,q) \), depending on whether \( I^{(n)} \) is a non-spacer or a spacer interval. We then show that for any \( I^{(m)} \), the application of \( \phi_t \) generates different heights. In the non-spacer case, we prove that for any vertical subinterval \( [a,b] \subseteq [0,1] \) and any horizontal interval \( I^{(n)} \), the ratio of the number of heights attained by \( \phi_t(I^{(m)}) \) within \( [a,b] \) (such that \( \phi_t(I^{(m)}) \) remains in the \( I^{(n)} \)-cylinder) to the total number of heights within the \( I^{(n)} \)-cylinder converges to the measure of \( [a,b] \). Similarly, for any \( [a,b) \subseteq [0,q) \), the ratio of the number of heights attained by \( \phi_t(I^{(m)}) \) within \( [a,b] \) in the \( I^{(n)} \)-cylinder to the total number of heights in \( I^{(n)} \) converges to \( \lambda([a,b]) / q \).  \\

From there, if $R=b(R)\times h(R)$ and $$b(R)=\bigcup_{i\in J}I^{(n_i)}$$ where $J$ is some index and $I^{(n_i)}$ is an interval in the $n_i$-th step of the staircase configuration then 

\[
\lim_{t\to\infty}\lambda(I^{(m)}\cap \phi_{-t}(R))=\lim_{t\to\infty}\sum_{i\in J}\lambda(\phi_t(I)\cap I^{(n_i)}\times h(R))\geq \mu(R)\lambda(I^{(m)}).
\]

For the classical staircase case we can calculate explicitly the intervals where the roof function is $1$ and $q$. If the first interval of definition has length $A$, \label{lenght} then the first spacer $s_2$ (remember that $s_1=0$) will have length $\frac{A}{2}$, for the third step we have to add $3$ spacers each of length $\frac{A}{6}$. In general, we have that the total length of all the spacers added after $k$ steps is $$
\sum_{n=1}^k \frac{A}{n!}\sum_{i=1}^n(i-1)=\sum_{n=1}^{k} \frac{A(n-1)}{2(n-1)!}=\frac{A}{2} + \sum_{n=2}^{k}\frac{A}{2(n-2)!}.  $$ Which converges to $\frac{A}{2} + \frac{Ae}{2}$. Therefore if we want  the interval $I$ to have length $1$ then we should have that $A=\frac{2}{3+e}$. \\

\begin{figure}
    \centering

\begin{tikzpicture}[scale=5]
  
  \draw[->] (0,0) -- (1.1,0) ;
  \draw[->] (0,0) -- (0,1.1) ;
  
  \draw[thick,blue] (0,1) -- (0.35,1);
  \draw[thick,blue] (0.35,0.8) -- (1,0.8);
  
  \filldraw[black] (0,1) circle (0.3pt) node[left] {1};
  \filldraw[black] (0.35,1) circle (0.3pt) node[above]{$\frac{2}{3+e}$};
  \filldraw[black] (0.35,0.8) circle (0.01pt);
  \filldraw[black] (1,0.8) circle (0.3pt) node[right] {1};
   \draw[ ->, dashed,red] (.1,0) -- (.1,.5);
  \draw[ ->, dashed,red] (.1,.5) -- (.1,1);
  \draw[ ->, dashed,red] (.28,0) --(.28,.5);
  \draw[ ->, dashed,red] (.28,.5) -- (.28,1);

  \draw[ ->, dashed,red] (.46,0) -- (.46,.4);
  \draw[ ->, dashed,red] (.46,.4) -- (.46,.8);
  \draw[ ->, dashed,red] (.16,0) --  (.16,.5); 
  \filldraw[black] (.1,0) circle (0.3pt) node[below left, font=\tiny]{x};
\filldraw[black] (.28,0) circle (0.3pt) node[below right, scale=.5]{$T(x)$};
\filldraw[black] (.46,0) circle (0.3pt) node[below, scale=.5]{$T^2(x)$};
\filldraw[black] (.16,0) circle (0.3pt) node[below ,scale=.5]{$T^3(x)$};
\filldraw[black] (0,0) node[below]{0};
\filldraw[black] (0,.8) circle(0.3pt) node[left]{$q$};
\filldraw[black] (1,0) circle (0.3pt) node[below]{1};

\end{tikzpicture}

   \label{fig:enter-label}
   \caption{Part of the orbit of the special flow built under $f$ and the classical staircase $T$.}
\end{figure}
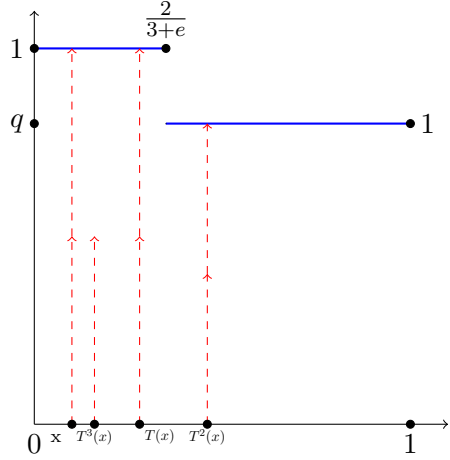

\begin{lemma} \label{lemma4.2}
There exists an $m_0$ depending solely on $q$ such that for every \mbox{$m\geq m_0$},  when $I^{(m)}$ be an $m$-level of the configuration of the classical  staircase $T$ and $R$ is a rectangle under the graph of $f$. Then if $\pi_x$ is the horizontal projection and     
     $b(R)=\pi_x(R)$ is the base of the rectangle $R$, then $$\lim_{t\to \infty} \lambda(\phi_t(I^{(m)})\cap \pi_x^{-1}(\pi_x(R))=\lambda(I^{(m)})\lambda((b(R)).$$ \\

\end{lemma}

\begin{proof}
For this proof, we will use the fact that the classical staircase transformation $T$ is mixing.\\

First we need to check that if $m$ is large, then for arbitrary $t$ $$\lambda(\pi_x(\phi_t(I^{(m)}))=\lambda(I^{(m)}).$$

    Suppose that for a level $I^{(m)}$ $$\lambda(\pi_x(\phi_t(I^{(m)})))<\lambda(I^{(m)}).$$

    This implies that there exist two different sublevels $I^{(n+m)}_{i}$ and $I^{(n+m)}_{j}$ and non negative real numbers $t,s$ such that $$\phi_{t-s}(I^{(n+m)}_{i})=\phi_t(I^{(n+m)}_{j}).$$ Moreover, $s\geq0$ must be smaller than $1$ or $q$.\\

    Observe that if we order upwards the sublevels in which $I^{(m)}$ is partitioned in the $(m+n)$-th step of the configuration of the staircase,  then $j<i$. This complication arises because under the intervals positioned beneath $1$, there may be accumulations of other intervals at varying heights. To illustrate this concept: If an interval lies below $q$, a single iteration of $\phi_q$ elevates this interval to the subsequent level in the staircase configuration. Conversely, if the interval lies beneath $1$, there exists a value $N_0>1$ such that $(N_0-1)q<1\leq N_0q$. This indicates that it takes several iterations (specifically, at most $N_0$ iterations) before the interval advances to the next level in the staircase configuration, therefore if two sublevels $I^{(n+m)}_i,I^{(n+m)}_j$ are in a distance smaller than $2N_0$ it follows that under the action of the flow $\phi_t$ we might encounter several instances $t_l$ in which 
    $$\pi_x(\phi_{t_l}(I_i^{(n+m)}))\cap \pi_x(\phi_{t_l}(I_j^{(n+m)}))\neq\emptyset.$$ 

We will demonstrate that if the initial interval is sufficiently small we will not find this problem.
\\
Observe that if $I^{(m)}$ is in the $m$-th step of the configuration of the staircase, when we consider the $m+1$ subintervals in which it is divided in the $(m+1)$-th step of the configuration of the staircase, then the distance between any two of those levels will be bounded from bellow by $h_m$, the height of the $m$-th step of the configuration. Since $h_m$ is an increasing sequence, then there exists an $m_0$ such that $h_{m_0}\geq M_0= 4N_0$, therefore for any $m\geq m_0$ and for any interval $I^{(m)}$ in the $m$-th step of the configuration of the staircase, it follows that the number of levels between any two subintervals $I^{(m+n)}_i$ and $I^{(m+n)}_j$ for any $n\geq 0$ will be bounded from bellow by $4N_0$ and therefore   
   $$\pi_x(\phi_{t}(I_i^{(n+m)}))\cap \pi_x(\phi_{t}(I_j^{(n+m)}))=\emptyset.$$
   
Continuing with the proof. Notice that for each $t$ there exists an $n\geq 0$ such that $\phi_t(I^{(m)})$ is scattered in $\frac{n!}{m!}$ subintervals. For each subinterval $I_j^{(m)}$ with $j=1,...,\frac{n!}{m!}$ there exists $l_j$ such that $\pi_x(\phi_t(I_j^{(m)}))=T^{l_j}(I_j^{(m)})$. From previous observation $$\pi_x(\phi_t(I^{(m)}))=\bigcup_{j=1}^{n!/m!}T^{l_j}(I_j^{(m)}).$$
Define $m(t)$ as the minimum of such $l_j$ and $M(t)$ as the maximum. We can assume that for large $t$ it holds true that
$$|\lambda(b(R)\cap T^{M(t)}(I^{(m)}))-\lambda(b(R))\lambda(I^{(m)})|\leq$$
$$|\lambda(b(R)\cap \pi_x(\phi_t(I^{(m)})))-\lambda(b(R))\lambda(I^{(m)})|\leq$$                         $$|\lambda(b(R)\cap T^{m(t)}(I^{(m)}))-\lambda(b(R))\lambda(I^{(m)})|.$$
    Since $T$ is mixing, this implies that 
    $$\lambda(b(R)\cap \pi_x(\phi_t(I^{(m)})))\to\lambda(b(R))\lambda(I^{(m)}).$$
\end{proof}

\begin{remark}
    
Let $m_0$ be the integer such that for every $m\geq m_0$ and every $I^{(m)}$ Lemma \ref{lemma4.2} holds. If $I^{(n)}$ is an interval in the $n$-th step of the configuration of the staircase with $n<m_0$, observe that $I^{(n)}$ is made up by 
$$I^{(n)}=\bigsqcup_{i=1}^{m!/n!}I^{(m)}_i.$$
Therefore for any rectangle $R$ it follows that
$$\lim_{t\to\infty}\lambda(\phi_t(I^{(n)})\cap \pi_x^{-1}(b(R)))=\sum_{i=1}^{m!/n!}\lambda(I^{(m)}_i)(\lambda(b(R))=\lambda(I^{(n)})\lambda(b(R)).$$
\end{remark} 

It is worth noting that when we consider an interval $I^{(n)}$ and observe how it evolves under the flow $\phi_t$, over time, this interval will break up into several horizontal segments scattered throughout the polygon. For the second part of the proof, our goal is to approximate the distribution of the heights achieved by these horizontal segments.\\

Our strategy will be to assume, without loss of generality, that the heights of the polygon are 1 and $q$. We then consider a discretization of the flow given by the iterates of $\phi_q$. By doing so, we aim to emulate the behavior of an irrational rotation and demonstrate that the heights of the horizontal segments follow a distribution similar to the set $ \{nq \mod{1} : n \in \mathbb{N}\} $.

\subsubsection{Behaviour of the first subcolumn}\label{Behaviour} 

The discussion that follows is crucial for delving into the second part of the proof, with the aim of convincing us that by discretizing the flow $ \phi_t $ through iterations of $ \phi_q $, applying $ \phi_{mq} $ to any level $ I^{(n)} $  yields an increasing sequence of heights contained in $ \{nq, n \in \mathbb{N} \mod{1}\} $ .\\

\textbf{Notation.} We say that $ x \in I $ \textit{lies beneath $ q $} if $ f(x) = q $. similarly, we say that $x\in I$ \textit{ lies beneath $1$} $ f(x) = 1 $.\\

Note that if a subinterval $ I^{(m)}$ lies beneath $ q $ at height zero, applying $ \phi_q $  resets its height to zero. Subsequent applications of $ \phi_q $  maintains this zero height as long as the subinterval remains beneath $ q $, until it eventually reaches a height of zero beneath $ 1 $. At this point, the application of $ \phi_q $  results in alternating heights: first $ q $, then $ 2q \mod{1} $, and so forth. Due to the construction of the classical staircase, the subinterval passes through the roof whose value is $ 1 $ twice consecutively before going to the spacer side.\\

Now, consider $ I^{(m)} $ as the last level of the $ m $-th step configuration of the staircase before it breaks into $ m+1 $ new subintervals $ I_j^{(m)} $ with $ j \leq m+1 $. We explain that the heights attained by  $ \phi_{nq}(I^{(m)}) $ with $m+1\leq n\leq h_m$ coincide with the heights attained  by $ \phi_{nq}(I^{(m)}_1),\phi_{(n-1)q}(I^{(m)}_1),\ldots,\phi_{(n-1-m)q}(I^{(m)}_1) $.\\

Let \( I^{(m)} \) be as defined in the previous paragraph. Observe that each subinterval \( I^{(m)}_j \) has \( j-1 \) spacers above it, all of which lie beneath \( q \) by construction. Consequently, applying \( \phi_q \) to \( I^{(m)}_j \) up to \( j-1 \) times ensures that \( \phi_{nq}(I^{(m)}_j) \) remains beneath \( q \) and has height zero for any \( n \leq j-1 \). Moreover, applying \( \phi_q \) exactly \( j \) times to \( I^{(m)}_j \) results in an image that lies beneath \( 1 \) and also has height zero.  \\

Thus, for any \( j \leq m \), the height of \( \phi_q(I^{(m)}_{j+1}) \) matches that of \( I^{(m)}_{j} \). This observation implies that the heights attained by successive iterations of \( \phi_{nq} \) on \( I^{(m)}_{j+1} \) are directly related to those attained by iterating \( \phi_{(n-1)q} \) on \( I^{(m)}_{j} \). Therefore, in the \( (m+1) \)-th step configuration of the staircase, the heights of the intervals \( I^{(m)}_j \) are fully determined by the heights attained by \( I^{(m)}_1 \).  We will soon see that this reasoning extends to arbitrarily large iterations. For now, it is clear this holds when \( n \leq h_m \). See Figures \ref{Heights1}, \ref{Heights2}, and \ref{Heights3} for a visualization of this concept.\\

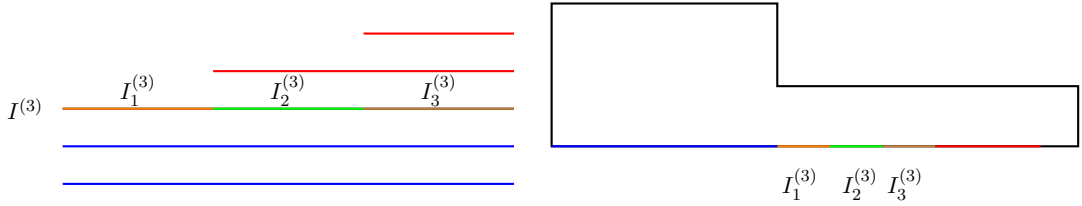
\begin{figure}
    \centering

\begin{tikzpicture}

\draw[blue, thick] (0,-2) -- (6,-2);
\draw[blue, thick] (0,-2.5) -- (6,-2.5);

\draw[thick] (0,-1.5) -- (2,-1.5);
\draw[thick] (2,-1.5) -- (4,-1.5);
\draw[thick] (4,-1.5) -- (6,-1.5);

\draw[thick,orange] (0,-1.5) -- (2,-1.5);
\draw[thick,green] (2,-1.5) -- (4,-1.5);
\draw[thick, brown] (4,-1.5) -- (6,-1.5);

\draw[red, thick] (2,-1) -- (6,-1);

\draw[red, thick] (4,-0.5) -- (6,-0.5);

\draw[thick] (6.5,-2) -- (13.5,-2) -- (13.5,-1.2) -- (9.5,-1.2) -- (9.5,-.1) -- (6.5,-.1) -- cycle;
\draw[thick, orange] (9.5,-2)--(10.2,-2);
\draw[thick, green] (10.2,-2)--(10.9,-2);
\draw[thick, brown] (10.9,-2)--(11.6,-2);
\draw[thick, blue] (6.5,-2)--(9.5,-2);
\draw[thick, red] (11.6,-2)--(13,-2);

\node[scale=.8] at (-.5,-1.5) {$I^{(3)}$};
\node[scale=.8] at (1,-1.25) {$I^{(3)}_1$};
\node[scale=.8] at (3,-1.25) {$I^{(3)}_2$};
\node[scale=.8] at (5,-1.25) {$I^{(3)}_3$};
\node[scale=.8] at (9.8,-2.5) {$I^{(3)}_1$};
\node[scale=.8] at (10.6,-2.5) {$I^{(3)}_2$};
\node[scale=.8] at (11.2,-2.5) {$I^{(3)}_3$};
\end{tikzpicture}
 \caption{Left: third step of the staircase configuration, blue levels are non spacers, red levels are spacers, the level $I^{(3)}$ is partinioned in the sublevels before it breaks into three levels.\\
 Right: Representation of $I^{(3)}$ in the $L$-shaped polygon. }
    \label{Heights1}
\end{figure}

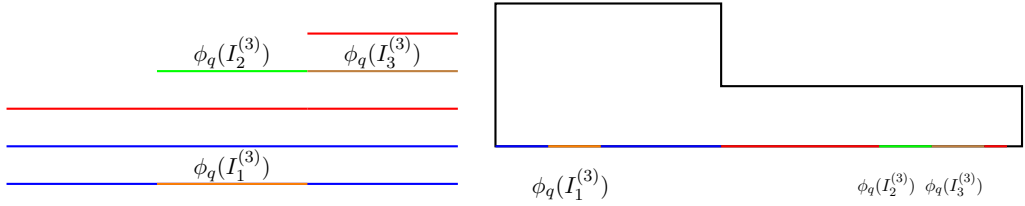
\begin{figure}
    \centering

\begin{tikzpicture}

\draw[blue, thick] (0,-2) -- (6,-2);
\draw[blue, thick] (0,-2.5) -- (6,-2.5);
\draw[orange, thick] (2,-2.5)--(4,-2.5);

\draw[red, thick] (0,-1.5) -- (2,-1.5);
\draw[red, thick] (2,-1.5) -- (4,-1.5);
\draw[red, thick] (4,-1.5) -- (6,-1.5);

\draw[thick, green] (2,-1) -- (4,-1);
\draw[thick, brown] (4,-1)--(6,-1);

\draw[red, thick] (4,-0.5) -- (6,-0.5);

\draw[thick] (6.5,-2) -- (13.5,-2) -- (13.5,-1.2) -- (9.5,-1.2) -- (9.5,-.1) -- (6.5,-.1) -- cycle;

\draw[thick, green] (11.6,-2)--(12.3,-2);
\draw[thick, brown] (12.3,-2)--(13,-2);
\draw[thick, blue] (7.9,-2)--(9.5,-2);
\draw[thick, orange] (7.2,-2)--(7.9,-2);
\draw[thick, blue] (6.5,-2)--(7.2,-2);

\node[scale=.8] at (3,-2.25) {$\phi_q(I^{(3)}_1)$};
\node[scale=.8] at (3,-.75) {$\phi_q(I^{(3)}_2)$};
\node[scale=.8] at (5,-.75) {$\phi_q(I^{(3)}_3)$};
\node[scale=.8] at (7.5,-2.5) {$\phi_q(I^{(3)}_1)$};
\node[scale=.6] at (11.7,-2.5) {$\phi_q(I^{(3)}_2)$};
\node[scale=.6] at (12.6,-2.5) {$\phi_q(I^{(3)}_3)$};

\draw[thick, red] (9.5,-2)--(11.6,-2);
\draw[thick, red] (13,-2)--(13.3,-2);
\end{tikzpicture}
 \caption{After one application of $\phi_q$ the heights are reseted to zero but the intervals spread in the staircase configuration. }
    \label{Heights2}
\end{figure}
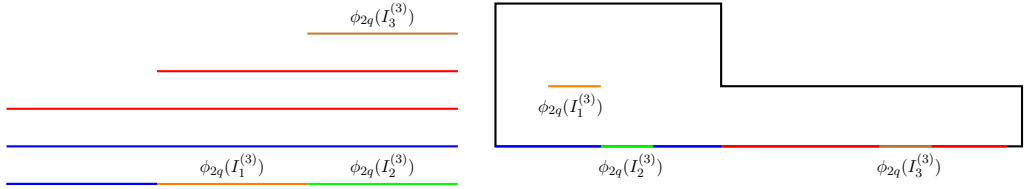
\begin{figure}
    \centering

\begin{tikzpicture}

\draw[blue, thick] (0,-2) -- (6,-2);
\draw[blue, thick] (0,-2.5) -- (2,-2.5);

\draw[thick, red] (0,-1.5)--(6,-1.5);
\draw[thick, blue] (0,-2.5)--(2,-2.5);

\draw[thick, green] (4,-2.5) -- (6,-2.5);
\draw[thick, brown] (4,-.5)--(6,-.5);

\draw[thick, orange] (2,-2.5) -- (4,-2.5);
\draw[red, thick] (2,-1)--(6,-1);
\draw[thick] (6.5,-2) -- (13.5,-2) -- (13.5,-1.2) -- (9.5,-1.2) -- (9.5,-.1) -- (6.5,-.1) -- cycle;
\draw[thick, orange] (7.2,-1.2)--(7.9,-1.2);
\draw[thick, green] (7.9,-2)--(8.6,-2);
\draw[thick, brown] (11.6,-2)--(12.3,-2);

\node[scale=.6] at (3,-2.25) {$\phi_{2q}(I^{(3)}_1)$};
\node[scale=.6] at (5,-2.25) {$\phi_{2q}(I^{(3)}_2)$};
\node[scale=.6] at (5,-.25) {$\phi_{2q}(I^{(3)}_3)$};
\node[scale=.6] at (7.5,-1.45) {$\phi_{2q}(I^{(3)}_1)$};
\node[scale=.6] at (8.3,-2.25) {$\phi_{2q}(I^{(3)}_2)$};
\node[scale=.6] at (12,-2.25) {$\phi_{2q}(I^{(3)}_3)$};
\draw[thick, blue] (6.5,-2)--(7.9,-2);
\draw[thick, blue] (8.6,-2)--(9.5,-2);
\draw[thick, red] (9.5,-2)--(11.6,-2);
\draw[thick, red] (12.3,-2)--(13.3,-2);
\end{tikzpicture}
 \caption{After two iteration of $\phi_q$ we obtain a first height different than zero }
    \label{Heights3}
\end{figure}

This behavior continues until the level $I^{(m)}_{m+1}$ reaches the final level in the \mbox{$(m+1)$-th} step of the staircase configuration, just before it splits into $m+2$ new sublevels $I^{(m)}_{m+1,l}$, where $l \leq m+2$. Using a similar argument, we find that the dynamics at this step of the configuration for the sublevels $I^{(m)}_{m+1,l}$ are governed by those attained by $I^{(m)}_{m+1,1}$. Furthermore, the heights reached by $I^{(m)}_{m+1,1}$ remain connected to the heights attained by the sublevels $I^{(m)}_1,...,I_{m}^{(m)}$, as explained in the previous paragraph. Therefore, to analyze the dynamics in the $(m+2)$-th step of the staircase configuration, it is sufficient to focus on $I^{(m)}_{1,1}$. Following this reasoning, we conclude that for any interval $I^{(m)}$, when considering its representation in the $(m+n)$-th step of the staircase configuration, the heights attained after several iterations, say $n$, coincide with the heights obtained by the set $${\phi_q(I^{(m)}_{1,...,1}), \dots, \phi_{nq}(I^{(m)}_{1,...,1})}.$$

From the previous paragraphs, we can draw the following conclusion. Let $I^{(m)}$ represent the final interval in the $m$-th step of the staircase configuration. Consider how this interval is represented in the $(m+n)$-th step of the staircase configuration. The heights reached by applying $\phi_q$ to $I^{(m)}$  $n$ times correspond to the heights achieved by the intervals $I^{(m+n)}_1, \phi_q(I^{(m+n)}_1), \dots, \phi_{nq}(I^{(m+n)}_1)$, where $I^{(m+n)}_1$ is the lowest level in the representation of $I^{(m)}$ within the $(m+n)$-th step of the staircase configuration, corresponding to $I^{(m)}_{1,...,1}$.\\

More precisely, if $k$ iterations of $\phi_q$ produce that $I^{(m)}$ lies in the $(m+n)$-th step of the staircase configuration, $I^{(m)}$ will be broken into several subintervals $I^{(m)}_J$ where $J$ is an index of length bounded by $(m+n-1)!/m!$, for one of such subintervals $I^{(m)}_{a_1,...,a_j}$ it will hold true that the height attained by $\phi_{kq}(I^{(m)}_{a_1,...,a_j})$ will coincide with the height attained by $\phi_{(K-N)q}(I^{(m)}_{1,...,1})$ where $$N=\sum_{i=1}^j|1-a_i|.$$

As $k\to \infty$ the staircase where $I^{(m)}$ will lie is the $(m+n)$-th step of the staircase configuration with $n\to \infty$, then the number of different heights also diverges. This result hints that, in the limiting process of iterations, and due to the equidistribution of $\{nq \mod 1 : n \in \mathbb{N}\}$,  the heights obtained will also equidistribute. 
\begin{remark}
   Observe that, although the heights obtained are described by the sets given in the previous paragraphs, there are many repetitions of heights, meaning that multiple subintervals can share the same height.

\end{remark}

\begin{remark} \label{remark1}
  In this procedure, we described the heights that can be obtained by iterating $\phi_q$ in the classical staircase configuration. However, if $\phi_t$ is a flow constructed under \emph{any} mixing staircase $T$, the result remains nearly the same. Specifically, we obtain an increasing sequence of attainable heights for an interval $I^{(m)}$ as the flow $\phi_t$ approaches infinity.
\end{remark}

We would like to prove that $$\lim_{t\to\infty} \lambda(\phi_t(I^{(m)})\cap R)\geq\lambda(I^{(m)})\lambda(b(R))\lambda(h(R)).$$
\\

We already have the first part, for the second part we would like that the concentration of heights contained in the vertical interval spanned by $R$ does not depend on $I^{(m)}$ but only on $h(R)$. For this purpose we need a notion of equitistribution.\\

As noted in \ref{Behaviour}, when examining the images of a subinterval from any $m$-step configuration of the staircase under the flow $\phi_t$, the subinterval breaks into further segments. Due to the presence of spacers arranged in a staircase-like pattern, this process results in a sequence of different heights. Our objective is to study the distribution of these heights.\\

The distribution of heights behaves differently depending on whether the intervals lie 
beneath $1$ or beneath $q$. Specifically, if an interval lies under $1$, every iteration 
of $\phi_q$ changes its height. In contrast, if the interval lies under $q$, it is mapped 
to another level of the staircase while preserving its height. Moreover, blocks of spacer 
intervals all share the same height. Thus, in order to analyze the distribution of heights 
generated by iterations of $\phi_q$, it is essential to distinguish between intervals 
lying beneath $1$ and those lying beneath $q$. For the next lemma, we introduce a 
convenient notation.

\begin{definition}
    Define $I_{ns}$ as the set of $x\in [0,1]$ such that $x$ is beneath 1, and define $I_s$ as the set of $x\in [0,1]$ such that $x$ is beneath $q$.
\end{definition}

\begin{definition}\label{Ht}
    Let $t \in \mathbb{R}$, and let $I,J$ be horizontal intervals, while $[a,b] \subseteq [0,1)$ 
    (or $[0,q)$, depending on whether we are in the non-spacer or spacer part) is a vertical interval. 
    We define the \emph{set of heights} as follows:
    \[
    H_t^{ns}(I,J,[a,b]) = 
    \left\{ z \in [a,b] \;:\; 
    \exists\, y \in I \text{ such that } 
    \pi_y(\phi_t(y)) = z \ \text{and}\ 
    \pi_x(\phi_t(y)) \in J \right\},
    \]
    when $J$ is contained in the non-spacer part. Similarly, we set
    \[
    H_t^{s}(I,J,[a,b]) = 
    \left\{ z \in [a,b] \;:\; 
    \exists\, y \in I \text{ such that } 
    \pi_y(\phi_t(y)) = z \ \text{and}\ 
    \pi_x(\phi_t(y)) \in J \right\},
    \]
    when $J$ is contained in the spacer part.
\end{definition}

  \begin{definition} \label{Hola}

         Let $I^{(m)}$ be a level in the $m$-step of the staircase configuration. We say that $I^{(m)}$ is $H$-equidistributed if, for any interval $[a,b] \subseteq [0,1]$ in the vertical direction,

 $$\lim_{t\to\infty} \frac{\textbf{card}(H_t^{ns}(I^{(m)},I_{ns},[a,b]))}{\textbf{card}((H_t^{ns}(I^{(m)},I_{ns},[0,1]))}=b-a,$$
and if, for any $[a,b] \subseteq [0,q)$,
 $$\lim_{t\to\infty} \frac{\textbf{card}(H_t^{s}(I^{(m)},I_s,[a,b]))}{\textbf{card}((H_t^{ns}(I^{(m)},I_s,[0,q]))}=\frac{b-a}{q}.$$
        
  \end{definition}

Achieving \( H \)-equidistribution in the spacer part is not straightforward. Although iterating \( \phi_q \) yields an increasing sequence of heights, understanding the precise distribution of these heights in the spacer part is crucial. While this might initially appear to obstruct \( H \)-equidistribution, the following result shows that the structured nature of these repetitions ensures that \( H \)-equidistribution is still achieved.

\begin{remark} \label{remark correspondence}
As noted in \ref{Behaviour}, the heights attained by $\phi_{kq}(I^{(m)})$ are determined by the heights of the iterates  
\[
\phi_{kq}(I^{(m)}_{1,\dots,1}), \quad \phi_{(k-1)q}(I^{(m)}_{1,\dots,1}), \quad \dots, \quad \phi_{q}(I^{(m)}_{1,\dots,1}),
\]
where the length of the sequence $1,\dots,1$ reflects the number of times $I^{(m)}$ has passed through the endpoint of different steps in the staircase configuration.\\

Among these iterations, there exist specific times $n^s_1, \dots, n^s_i$ at which $\phi_{n^s_l q}(I^{(m)}_{1,\dots,1})$ falls into a spacer interval, and times $n_1^{ns}, \dots, n_j^{ns}$ at which $\phi_{n_j^{ns} q}(I^{(m)}_{1,\dots,1})$ falls into a non-spacer interval.\\

From the observations in \ref{Behaviour}, we note that the heights attained by subintervals of $I^{(m)}$ after multiple iterations of $\phi_q$ are determined by those achieved by $I^{(m)}_{1,\dots,1}$. Moreover, applying $\phi_q$ to a spacer interval preserves its height, whereas applying it to a non-spacer interval changes its height. This leads to the following conclusions:\\

\begin{itemize}
    \item For each time $n_l^s$, there is an associated height $h_l^s$ such that these heights correspond to subintervals of $I^{(m)}$ that, after $k$ iterations of $\phi_q$, lie in cylinders whose base is a spacer interval.
    \item Similarly, for each time $n_l^{ns}$, there is an associated height $h_l^{ns}$ such that these heights correspond to subintervals of $I^{(m)}$ that, after $k$ iterations of $\phi_q$, lie in cylinders whose base is a non-spacer interval.
\end{itemize}

The heights $h_l^s$ and $h_l^{ns}$ correspond to the values attained by $\phi_{n_l^sq}(I^{(m)}_{1,\dots,1})$ and $\phi_{n_l^{ns}q}(I^{(m)}_{1,\dots,1})$. Since the heights of the iterates of $\phi_q$ for $I^{(m)}_{1,\dots,1}$ remain unchanged while passing through spacer blocks, multiple times $n_l^s$ can correspond to a single height. In contrast, for the times $n_l^{ns}$, the mapping to heights is injective.
\end{remark}

\begin{proposition}\label{propomiau}
    For any subinterval $I^{(m)}$ is $H-$equidistributed.   
\end{proposition}     
\begin{proof}

First, consider the non-spacer case. To establish the claim, we assume that \( I^{(m)} \) is initially at the uppermost level in the \( m \)-step configuration of the staircase transformation before breaking into \( m+1 \) new subintervals. If this is not the case, we can shift the interval using the flow until it reaches this uppermost level. This adjustment preserves the connectedness of the interval, ensuring that during the shifting process, it attains only one height.\\

Using the ideas developed in \ref{Behaviour}, we observe that the heights attained by \( \phi_{kq}(I^{(m)}) \) coincide with those attained by \( I^{(m)}_{1,\dots,1} \). Since the heights of \( I^{(m)}_{1,\dots,1} \) remain unchanged while passing through spacer blocks, it follows that all the heights achieved by \( I^{(m)}_{1,\dots,1} \) after multiple iterations of \( \phi_q \) are found in the non-spacer side while in the spacer we find only a subsequence of such heights. From the observations in \ref{remark correspondence}, we conclude that these heights are attained by subintervals that, after multiple iterations, lie in cylinders whose base is non-spacer. The equidistribution of \( nq \mod 1 \) then establishes the first part of the proof.\\

For the second part of the proof, we analyze the heights attained by the spacer subintervals, which correspond to the blocks of spacer levels in the staircase configuration. First, observe that, by \ref{remark correspondence}, the heights of the intervals of \( I^{(m)} \) that, after iterations of \( \phi_q \), lie in cylinders whose base is a spacer interval correspond to the heights \( h^{ns}_l \). Thus, it suffices to study the distribution of heights achieved by \( I^{(m)}_{1,\dots,1} \) after multiple iterations. To achieve this, we construct two sequences as follows:\\

Let $m_1$ be the smallest integer such that 
$(m_1 - 1)q < 1 < m_1 q$. Note that\\ $m_1 q \mod 1 < q$. Let $k_1'$ be the smallest integer such that $$(k_1' - 1)(m_1 q \mod 1) < 1 < k_1'(m_1 q \mod 1),$$ and define $k_1 = k_1' m_1$. Also, $k_1 q \mod 1 < q$.\\

Next, let $m_2'$ be the smallest integer such that $$(m_2' - 1)(k_1 q \mod 1) < 1 < m_2' (k_1 q \mod 1),$$ and define $m_2 = m_2' k_1$.\\

Then, let $k_2'$ be the smallest integer such that $$(k_2' - 1)(m_2 q \mod 1) < 1 < k_2'(m_2 q \mod 1),$$ and define $k_2 = k_2' m_2$.\\

In general, suppose $k_n$ is defined. Let $m_{n+1}'$ be the smallest integer such that $$(m_{n+1}' - 1)(k_n q \mod 1) < 1 < m_{n+1}' (k_n q \mod 1),$$ and define $m_{n+1} = m_{n+1}' k_n$.\\

Let $k_{n+1}'$ be the smallest integer such that $$(k_{n+1}' - 1)(m_{n+1} q \mod 1) < 1 < k_{n+1}'(m_{n+1} q \mod 1),$$ and define $k_{n+1} = k_{n+1}' m_{n+1}$.\\

It is clear that both sequences satisfy $k_n q \mod 1 < q$ and $m_n q \mod 1 < q$. Moreover, the sequence $k_n q \mod 1$ describes the heights attained by the spacer intervals under the iterations of $\phi_q$.\\

From this, we can convince ourselves that if
$a_n$ is the sequence such that\\
$a_n q \mod 1 < q$, then since $\{nq \mod 1\}$ equidistributes in $[0,1)$, it follows that since $\{a_n q \mod 1\}$ describes the orbit of $0$ of the first return map of the irrational rotation by $q$ to $[0,q)$, this first return map $g$ is again an irrational rotation, in particular all its iterates  $g^k$ are uniquely ergodic, therefore its orbits equidistribute in $[0,q)$. \\

From the construction of the sequence $k_i$, we can observe that $k_i = a_{2i}$, since the odd-indexed terms correspond to $m_i$. In other words, the sequence $k_nq \mod 1$ describes the orbit of the second return map to $[0,q)$, which is $g^2$ in the last paragraph. In particular, it is uniquely ergodic, and the orbit of zero will equidistribute. Consequently,
$$
\lim_{n \to \infty} \frac{\textbf{card} \left(\{ m \leq n \mid k_m q \mod 1 \in [a,b] \} \right)}{n} =\frac{b-a}{q}
$$
for every \( [a,b] \subseteq [0,q) \). \\

As mentioned, the analysis in the spacer case is complicated since multiple intervals can share the same height. According to last paragraph, the sequence $k_n q \mod 1$ describes the heights attained by the spacer intervals. In the $m$-th step of the staircase configuration, spacers appear in blocks, following \ref{Behaviour} if we apply sufficiently many times $\phi_q$ in such a way that $I^{(m)}_1$ goes  its way through one $m(q)$-th  step of the staircase configuration,  all intervals within a block of spacers share the same height. Rather than considering individual intervals, we focus on these blocks as a whole. \\

Arrange the blocks of spacers by order or appearance upwards in a staircase configuration. 
Consider an ennumeration $\{d(l)\}_{l\in \mathbb{N}}$ of that appearance such that $d(l)$ describes the number of intervals that  the $l$-th spacer block in order of appearance has. For any $n$, we examine the sequence of blocks containing $n$ intervals. It can be observed that the appearance of these blocks follows a linear pattern. Specifically, for each $n$, there exists  values $a_n$ and $b_n$ such that for any $l$, the  block of spacers $d(b_n+a_n l)$ contains exactly $n$ intervals. Consequently, for each $n$, the sequence $n_l = k_{b_n+a_n l}$ is such that the set $\{n_l q \mod 1 : l \in \mathbb{N}\}$ accurately describes the heights of spacer intervals that appear in blocks of $n$ elements. \\

    We can give a description of the sequences $n_j$. For the case $n=1,2$ the sequences are given by $1_j=1+3j$ and $2_j=2+3j$ respectively. \\
    
    Let $n$ be any other number.
    Consider $m(n)$ the first step in the staircase configuration such that there exists a block of spacers with $n$ intervals. For $m(n)$ there will be $m(n)-1$ new numbers $c_i$ such that there exists a block with $c_i$ spacer intervals such that these numbers did not appear in the $(m(n)-1)$-th step of the staircase configuration. Order them from bottom to top. Let $F(n)$ be the first time $n$ appears in the list of blocks. If $n$ is not $c_{m(n)-1}$ then $$n_j=F(n)+ j\frac{m(n)!}{2},$$ and if $n$ is $c_{m(n)-1}$ then $$n_j=F(n)+j\frac{(m(n)+1)!}{2}.$$

Continuing with the proof; just as before,  using the fact that the first return map of a rotation by $q $ to the subinterval $[0,q)$ is totally ergodic, and observing that for any $n$ the sequence $n_jq\mod 1$ is the orbit of $g^{2b_n}(0)$ under the transformation $g^{2a_n}$ we see that the sequence of heights attained by blocks with $n$ intervals equidistribute in $[0,q)$.\\

Define $P_n$ as the probability of finding a block with $n$ elements. Then,
$$\lim_{n\to\infty} \frac{\textbf{card}(H_{nq}^{s}(I^{(m)}, I_{s,} [a,b]))}{\textbf{card}(H_{nq}^{s}(I^{(m)}, I_{s},[0,q)])}=\sum_{n\in\mathbb{N}}\left(\frac{b-a}{q}\right)P_n=\frac{b-a}{q}.$$

Observe that $H$-equidistribution for $t\to\infty$ follows directly from the $H$-equidistribution when $nq\to\infty$,  since for large $t$ the variation when considered the discrete and continuous version converges to zero.
The proposition follows.
\end{proof}

\begin{remark}
    Observe that $H$-equidistribution is not enough for approximating the measure $\lambda(\phi_t(I^{(m)}\cap R)$,  since there might be a \textit{diagonal} behavior, see Figure \ref{diagonal}. Nevertheless if we prove something similar, but instead of considering $J=I_{ns}$ or $J=I_s$ to be an arbitrary interval of the staircase configuration $I^{(n)}$, and prove equidistribution on the heights in those cylinders, then the approximation of  $\lambda(\phi_t(I^{(m)}\cap R)$ is achievable.
\end{remark}
\begin{figure}
    
\begin{tikzpicture}[scale=.5]
   
    \draw[thick, black] (0,0) -- (8,0) -- (8,3) -- (3,3) -- (3,5) -- (0,5) -- cycle;

    \foreach \i in {1,...,30} {
        \draw[red, thick] (\i/10, 5-\i*5/30) -- (\i/10 + 0.1, 5-\i*5/30);
    }

    \foreach \i in {1,...,60} {
        \draw[red, thick] (3+\i*5/60, 3-\i/20) -- (3+\i*5/60 + 0.1, 3-\i/20);
    }
\draw[blue,thick] (1,1)--(1.5,1)--(1.5,1.5)--(1,1.5)-- cycle;
\end{tikzpicture}
\centering
   
    \caption{A diagonal behavior causes that the intervals miss several rectangles.}
    \label{diagonal}
\end{figure}
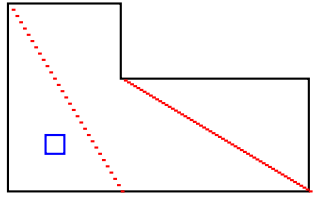
\begin{definition}
     Let $I^{(m)}$ be a level in the $m$-step of the staircase configuration. We say that $I^{(m)}$ is $\overline{H}$-equidistributed if, for any interval $[a,b] \subseteq [0,1]$ in the vertical direction, and any interval of the staircase configuration $I^{(n)}$

 $$\lim_{t\to\infty} \frac{\textbf{card}(H_t^{ns}(I^{(m)},I^{n},[a,b]))}{\textbf{card}((H_t^{ns}(I^{(m)},I^{n},[0,1]))}=b-a,$$
and if, for any $[a,b] \subseteq [0,q)$,
 $$\lim_{t\to\infty} \frac{\textbf{card}(H_t^{s}(I^{(m)},I^{(n)},[a,b]))}{\textbf{card}((H_t^{ns}(I^{(m)},I^{(n)},[0,q]))}=\frac{b-a}{q}.$$
\end{definition}
\begin{lemma}
    \label{H2equidistribution}
    For any $I^{(m)}$ it is $\overline{H}$-equidistributed.
\end{lemma}
\begin{proof}

Using the \( H \)-equidistribution stated in Proposition \ref{propomiau} and Lemma \ref{lemma4.2}, we establish that for any intervals \( I^{(m)} \) and \( I^{(n)} \), the heights attained by \( \phi_t(I^{(m)}) \) become equidistributed within the \( I^{(n)} \)-cylinder.  \\

We begin by analyzing the case where \( I^{(n)} \) lies in the non-spacer region. Without loss of generality, we assume \( m \geq n \). Considering a discretization of the flow given by iterates of \( q \), the evolution of \( \phi_{kq}(I^{(m)}) \) through the staircase causes it to break into subintervals. Over multiple iterations, some of these subintervals return to the \mbox{\( I^{(n)} \)-cylinder} due to the staircase-like structure. This implies the existence of subintervals \( I^{(m)}_{h_1}, \dots, I^{(m)}_{h_j} \) such that every \( \phi_{kq}(I^{(m)}_{h_i}) \)  lies within the \( I^{(n)} \)-cylinder. \\

These subintervals, in turn, belong to specific subcylinders of \( I^{(n)} \), denoted by \( I^{(n)}_{g_1}, \dots, I^{(n)}_{g_j} \). Among the different heights achieved by these subintervals, there exists a \textit{leading} interval, meaning that if the heights are given by \( n_iq \mod 1 \), one interval attains the maximum height $n_{i_0}q \mod 1$. More precisely, we can label the subintervals as \( I^{(m)}_1, \dots, I^{(m)}_j \) such that \( I_1^{(m)} = I^{(m)}_{i_0} \), and the height of each \( I_l^{(m)} \) coincides with that of \( \phi_{-N_{1,l}q}(I^{(m)}_1) \), where \( N_{1,l} \) represents the distance between \( l \) and \( 1 \) in this enumeration.
\\

By iterating this process, we observe that if $\phi_{Kq}$ causes $I^{(m)}$ to reach the the \mbox{\((m+d)\)-th} step of the staircase configuration, an even smaller portion of \( I^{(m)} \) will lie within the \mbox{\( I^{(n)} \)-cylinder}. More precisely, there exist subintervals \( I^{(m)}_{J_1}, \dots, I^{(m)}_{J_k} \), where each index \( J_i \) has length bounded by \( \frac{(m+d-1)!}{m!} \).  

Similarly, there are subcylinders over \( I^{(n)}_{J'_1}, \dots, I^{(n)}_{J'_k} \), where each \( J'_i \) has length bounded by \( \frac{(n+d-1)!}{n!} \). For each \( J_i \), there exist subsubintervals \( I^{(m)}_{J_i,i(1)}, \dots, I^{(m)}_{J_i,i(p)} \) that lie within some subcylinder \( I^{(n)}_{J'_i} \). Among these subsubintervals, one achieves the greatest height within the subcylinder, which we denote as the \textit{leading} subsubinterval \( I^{(m)}_{J_i,i(j_0)} \). \\

Following the reasoning from the previous paragraph, we can enumerate these subsubintervals as \( I^{(m)}_{J_i,1}, \dots, I^{(m)}_{J_i,p} \) so that \( I^{(m)}_{J_i,1} = I^{(m)}_{J_i,i(j_0)} \). The heights achieved by these subsubintervals satisfy  
\[
\pi_y(\phi_{Kq(} I^{(m)}_{J_i,l})) = \pi_y(\phi_{(K-N_{1,l})q}(I^{(m)}_{J_i,1})),
\]
where \( N_{1,l} \) denotes the distance between \( l \) and \( 1 \) in this enumeration.

Next, considering the heights of the leading subsubintervals across different subcylinders \( I^{(m)}_{J_i,1} \), we apply the previous argument. If multiple subsubintervals attain the maximum height, we order them from left to right according to their appearance in the subcylinders of \( I^{(n)} \). We then relabel the leftmost interval with the leading height as \( I^{(m)}_{1,1} \) and establish an enumeration such that the height of \( I^{(m)}_{j,1} \) satisfies  
\[
\pi_y(\phi_{Kq}(I^{(m)}_{j,1})) = \pi_y(\phi_{(K-N_{1,j})q}(I^{(m)}_{1,1})),
\]
where \( N_{1,j} \) represents the distance between \( j \) and \( 1 \). Extending this enumeration within each subcylinder, we ensure that for each subsubinterval \( I^{(m)}_{i,j} \),  
\[
\pi_y(\phi_{Kq}(I^{(m)}_{i,j})) = \pi_y(\phi_{(K-N_{1,j}+ N_{1,i})q}(I^{(m)}_{1,1})).
\]
See Figure \ref{cylinder} for a toy model illustrating this behavior.\\

With further iterations of \( \phi_q \), the measure of the portion of \( I^{(m)} \) that lies within the \( I^{(n)} \)-cylinder converges to \( \lambda(I^{(m)}) \lambda(I^{(n)}) \), as established in Lemma~\ref{lemma4.2}. Moreover, these portions become equidistributed among the subcylinders \( I^{(n)}_J \), where \( J \) is an index of a specified length that depends on the number of iterations of \( \phi_q \). The heights attained within these subintervals follow the same pattern as before, forming an increasing sequence of the form  
\[
kq, (k+1)q, \dots, (k+r)q \mod 1.
\]  
As \( q \to \infty \), the length of these sequences increases, leading to the equidistribution of heights within the cylinder generated by \( I^{(n)} \).\\

If \( I^{(n)} \) is a spacer interval, we observe that, due to the staircase-like pattern, the behavior is analogous to the non-spacer case. After multiple iterations of \( \phi_q \), the subintervals forming \( I^{(m)} \) distribute across subcylinders of \( I^{(n)} \), attaining different heights. The only distinction from the non-spacer case is that these heights do not differ by one iteration of $q \mod 1$ as in the non-spacer case, rather they follow  the subsequence described in Proposition \ref{propomiau}, which as proved there, ensures their equidistribution within the \( I^{(n)} \)-cylinder.

\end{proof}
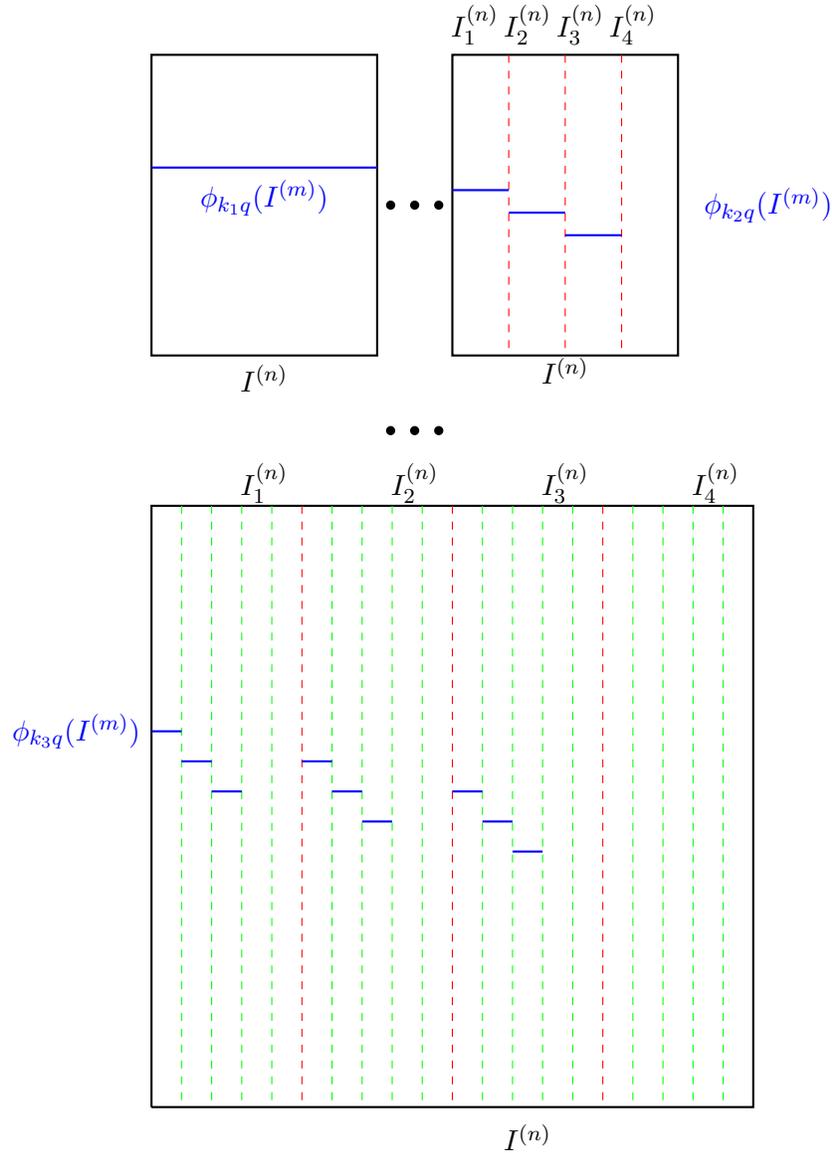
\begin{figure}
    \centering

\begin{tikzpicture}
    \draw[black, thick] (0,0)--(3,0)--(3,4)--(0,4)-- cycle;
    \node[black] at (1.5,-.3) {$I^{(n)}$};
    \draw[blue, thick] (0,2.5)--(3,2.5);
 \node[blue] at (1.5,2.1) {$\phi_{k_1q}(I^{(m)})$};
    \draw[black, thick] (4,0)--(7,0)--(7,4)--(4,4)-- cycle;
    \node[black] at (5.5,-.2) {$I^{(n)}$};
    \draw[red, dashed] (4.75,4)--(4.75,0);
    \draw[red, dashed] (5.5,4)--(5.5,0);
    \draw[red, dashed] (6.25,4)--(6.25,0);
    \draw[blue, thick] (4,2.2)--(4.75,2.2);
     \draw[blue, thick] (4.75,1.9)--(5.5,1.9);
      \draw[blue, thick] (5.5,1.6)--(6.25,1.6);
      \node[scale=3] at (3.5,2) {...};
      
      \node[black] at (4.3,4.4) {$I^{(n)}_1$};
      \node[black] at (5,4.4) {$I^{(n)}_2$};
      \node[black] at (5.7,4.4) {$I^{(n)}_3$};
      \node[black] at (6.4,4.4) {$I^{(n)}_4$};
      \node[scale=3] at (3.5,-1) {...};

    \draw[black, thick] (0,-10)--(8,-10)--(8,-2)--(0,-2)-- (0,-10);
    \foreach \i in {1,2,3} {
    \draw[red,dashed] (\i*2,-2)--(\i*2,-10);
    };
      \foreach \i in {0,1,2,3} {
      \foreach \j in {1,2,3,4} {
      \draw[green, dashed] (\i*2+\j*.4,-2)--(\i*2+\j*.4,-10);
      }
      };
      \foreach \i in {0,1,2}{

    \foreach \j in {0,1,2}{

\draw[blue, thick] (\i*2 +\j*.4,-5-\i*.4-\j*.4)--(\i*2+\j*.4+.4, -5-\i*.4-\j*.4);
    }
      }
      \node[black] at (5,-10.4) {$I^{(n)}$};

      \foreach \i in {1,2,3,4}{
      \node[black] at (-.5+\i*2,-1.7) {$I_{\i}^{(n)}$};
      }
     \node[blue] at (-1,-5) {$\phi_{k_3q}(I^{(m)})$};
     \node[blue] at (8.2,2) {$\phi_{k_2q}(I^{(m)})$};
\end{tikzpicture}
 \caption{Toy model of the behavior of $\phi_{nq}(I^{(m)})$ in the $I^{(n)}$-cylinder.}
    \label{cylinder}
\end{figure}
Now we can observe that Lemma  \ref{lemma4.2} and Lemma \ref{H2equidistribution} imply the following:
\begin{theorem} \label{teo4.14}
    Let $T$ represent the classical staircase transformation, and let $f$ be a $2$-wise constant function characterized by roof values $p$ and $q$ that are rationally independent. If all the spacers of the transformation $T$ are positioned beneath the roof with value $q$, then the flow constructed under $f$ respecting $T$ is strongly mixing.
\end{theorem}

\begin{proof}
    
    According to Lemma  \ref{LemaUlci} we need to prove that for any $\epsilon,\delta>0$ and for every rectangle $R$ there is a $t_0>0$ such that for every $t\geq t_0$ the exists a partition $\eta(t)$ such that:
    \\
    \begin{enumerate}
        \item $\lambda(\eta(t))>1-\delta$ and $Mesh(\eta(t))\leq \delta$.
        \\
        \item For every $I\in \eta(t)$ is follows that $\mu(I\cap \phi_{-t}(R))\geq (1-\epsilon)\mu(R)\lambda(I).$
    \end{enumerate}

    The partitions we construct depend on $\delta$, while $t_0$ is determined by $\epsilon$.\\

The natural partitions under consideration stem from the $m$-steps within the staircase configuration. Referring to \ref{lenght}, if $A$ denotes the length of $I^{(1)}$, then for any subinterval $I^{(m)}$ in the $m$-step, we have:
    
    \begin{equation}\label{yeroconvergence}
    \lambda(I^{(m)})=\frac{A}{(m)!}\to0.
     \end{equation}

   Moreover,  since all the intervals in this choice of partition have the same measure, then $Mesh$ of any of this partition will converge to zero in this way \ref{yeroconvergence}.\\

    On the other hand, we observe that the measure of the covers converges to 1 by construction. Therefore, there exists an $m$ such that $\lambda(I^{(m)})<\delta$ and simultaneously $h_m\lambda(I^{(m)})<1-\delta$ and also $m\geq m_0(q)$, where $h_m$ represents the number of intervals in the $m$-step of the transformation and $m_0(q)$ is the constant given in Lemma \ref{lemma4.2}. This satisfies the initial requirement.\\

  Without loss of generality we may assume that $b(R)$ is one interval $I^{(n)}$ of the $n$-th step of the staircase configuration for some $n$.  For the second requirement, we need to consider three cases:
\begin{enumerate}
    \item $R$ is inside the non-spacer part. Then
    $$\lambda(\phi_t(I^{(m)})\cap R)\to \lambda(\phi_t(I^{(m)})\cap \pi_x^{-1}(b(R)))\frac{\textbf{card}(H_t^{ns}(I^{(m)},I^{(n)},h(R)))}{\textbf{card}((H_t^{ns}(I^{(m)},I^{(n)},[0,1]))}.$$

    By Lemma \ref{lemma4.2},
    $$\lambda(\phi_t(I^{(m)})\cap \pi_x^{-1}(b(R)))\to \lambda(I^{(m)})\lambda(b(R)).$$

    And by Lemma \ref{H2equidistribution} we have that 
    $$\frac{\textbf{card}(H_t^{ns}(I^{(m)},I^{(n)},h(R)))}{\textbf{card}((H_t^{ns}(I^{(m)},I^{(n)},[0,1]))}\to\lambda(h(R)).$$

    This in particular implies that
    $$\lambda(I^{(m)}\cap \phi_{-t}(R))\to \lambda(I^{(m)})\lambda(b(R))\lambda(h(R))=\lambda(I^{(m)})\mu(R).$$

    \item $R$ is inside the spacer part of the polygon. Then by Lemma \ref{H2equidistribution},
    $$\frac{\textbf{card}(H_t^{ns}(I^{(m)},I^{(n)},h(R)))}{\textbf{card}((H_t^{ns}(I^{(m)},I^{(n)},[0,1]))}\to\frac{\lambda(h(R))}{q}.$$
    Then
    $$\lambda(I^{(m)}\cap \phi_{-t}(R))\to \frac{\lambda(I^{(m)})\mu(R)}{q}>\lambda(I^{(m)})\mu(R).$$

    \item $R$ has a part inside the spacer and non-spacer part. Let $R_1$ be the part inside the non-spacer part and $R_2$ be the part inside the spacer part. Observe that $h(R_1)=h(R_2)=h(R)$ and $b(R_1)+b(R_2)=b(R)$. Suppose that $b(R_1)=I^{n_1}$ and $b(R_2)=I^{(n_2)}$ By Lemma \ref{lemma4.2} and Lemma \ref{H2equidistribution} we have that 
    $$\lambda(I^{(m)}\cap \phi_{-t}(R))=\lambda(\phi_t(I^{(m)})\cap \pi_x^{-1}(b(R_1)))\frac{\textbf{card}(H_t^{ns}(I^{(m)},I^{(n_1)},h(R)))}{\textbf{card}((H_t^{ns}(I^{(m)},I^{(n_1)},[0,1]))}+$$
    $$\lambda(\phi_t(I^{(m)})\cap \pi_x^{-1}(b(R_2)))\frac{\textbf{card}(H_t^{ns}(I^{(m)},I^{(n_2)},h(R)))}{\textbf{card}((H_t^{ns}(I^{(m)},I^{(n_2)},[0,1]))}\to \lambda(I^{(m)})\lambda(b(R_1))\lambda(h(R_1))+$$
    $$\lambda(I^{(m)})\lambda(b(R_2))\frac{\lambda(h(R_2))}{q}>\lambda(I^{(m)})\mu(R).$$
\end{enumerate}

Observe that in all three cases,
$$\lim_{t\to\infty}\lambda(I^{(m)}\cap \phi_{-t}(R))\geq\lambda(I^{(m)})\mu(R).$$

In particular, for $\epsilon>0$ there exists a $t_0$ such that for any $t\geq t_0$ we have that
$$\lambda(I^{(m)}\cap \phi_{-t}(R))\geq (1-\epsilon)\lambda(I^{(m)})\mu(R).$$

To conclude the proof, we state that the $t_0$ constructed above will be useful, and the covers that serve for this purpose will consist of the intervals for the $m$-step configuration of the staircase, where the $m$ is the one that is useful for the $\delta$. With this, and in virtue of Lemma \ref{LemaUlci}, we conclude the proof.

\end{proof}

 \section{General Mixing Staircase} \label{gen}
The strategy for demonstrating that the flow constructed under a $2$-wise constant function, where its two heights are rationally independent values, over any mixing staircase transformation is analogous to the approach used for the classical staircase. In this section, we will outline the minor adjustments required from the previous procedure and refer to the proofs there. \\

We will once more rely on Lemma  \ref{LemaUlci}. The choice of partitions will be based on the natural partitions within the interval defined by the staircase transformation, with the specific partition determined by $\delta$ as per Lemma \ref{LemaUlci}, while the time $t_0$ will be contingent on $\epsilon$ and the rectangle.\\

In \cite{CS1} and \cite{CS2}, a sufficient condition is provided to characterize mixing staircase transformations and rank-one transformations.\\

Continuing with the previous procedure, our initial step is to demonstrate that:

\begin{lemma} \label{lemma5.1}
   There exists an $m_0$ such that for any $m \geq m_0$, for any level $I^{(m)}$ in the $m$-th step of the staircase configuration of any mixing staircase transformation $T$, and for any rectangle $R$ under the graph of $f$, the following holds. Let $\pi_x$ denote the horizontal projection, then:
   $$\lim_{t\to \infty} \lambda(\phi_t(I^{(n)})\cap \pi_x^{-1}(b(R)))=\lambda(I^{(n)})\lambda((b(R)).$$ \\
    
    Where $b(R)$ is the basis of the rectangle $R$.
\end{lemma}
\begin{proof}
    The proof is essentially the same as in the case of the Classical staircase. \\

    Consider an interval $I^{(n)}$ such that for some subintervals $I^{(n+m)}_i, I^{(n+m)}_j$ such that for some non negative real numbers $t,s $$$\phi_{t-s}(I^{(n+m)}_i)= \phi_t(I^{(n+m)}_j).$$
    Consider $N_0$ such that $(N_0-1)q<1\leq N_0q$, just as in Lemma \ref{lemma4.2}  the only problem that can arise is when there are two subintervals in the $(m+n)-$th step of the staircase configuration such that their distance is smaller than $2r_1N_0$ (remember that $\{r_n\}$ is the cutting sequence for $T$). This is due to the fact that there can be accumulations in the $r_1$ intervals that are in the non-spacer blocks. Since $h_m$ is an increasing sequence, there exists an $m_0$ such that $h_m\geq 4r_1 N_0$ for any $m\geq m_0$. Therefore for any $n$ the  distance between any two sublevels $I^{(n+m)}_i$ and $I^{(n+m)}_j$ of $I^{(m)}$ in the $(m+n)$-th step of the staircase configuration is greater than $4r_1N_0$. The argument to conclude the proof is the same as in Lemma \ref{lemma4.2}.

\end{proof}

\begin{remark}
    Noticeably, for this proof, we only necessitate the divergence of the sequence $h_n$ associated with the rank-one transformation.  Therefore, assuming the transformation is mixing, the second requirement implies that as the flow approaches infinity, considering the transformation $T$ suffices instead. Since $T$ is mixing, we can attain the convergence outlined in Lemma \ref{lemma5.1}.
\end{remark}

Just as in the classical staircase case, any mixing staircase exhibits a staircase-like pattern, even if the cutting sequence is not given by \( r_n = n \). Consequently, multiple iterations of \( \phi_q \) applied to an interval \( I^{(m)} \) cause it to break into several subintervals. Due to the staircase-like structure of the spacer intervals, the resulting heights exhibit the same behavior as in the classical case.  

Another way to interpret this phenomenon is by considering the \((m+n)\)-th step of the staircase configuration. At this stage, multiple copies of the \( m \)-th step appear, separated by spacer intervals. These spacer intervals introduce \textit{delays} in the heights of the newly formed subintervals, meaning that the new heights differ by one, as discussed in the previous section. Our goal is to prove that these heights equidistribute.  

Following the notation in Definition \ref{Ht} and  Definition \ref{Hola} we want to prove that:

\begin{proposition} \label{propo5.4}
      Any subinterval $I^{(m)}$ is $H-$equidistributed.  
\end{proposition}

\begin{proof}
 Following the same reasoning as in the proof of Proposition \ref{propomiau}, we note that since \( T \) is a staircase transformation, the introduction of new spacers at each step of the staircase configuration results in the behavior described in \ref{Behaviour}. Specifically, after several iterations of \( \phi_q \), the heights attained by an interval \( I^{(m)} \) correspond to those attained by the lowest interval in its representation at the \( (m+n) \)-th step of the staircase configuration. Consequently, when the $k$ iterations of  $\phi_q$ goes to infinity the sequence of heights eventually approaches the set \( \{nq \mod 1\}, a_k\leq n\leq b_k \) where $b_k-a_k\to\infty$. Moreover, in the non-spacer part of the polygon, each height is attained by exactly one interval, which establishes \( H \)-equidistribution in the non-spacer case.\\

For the second part, following Proposition \ref{propomiau}, we construct a sequence \( a_n \) that describes the heights such that \( a_nq \mod 1 < q \). Since the non-spacer blocks contain \( r_1 \) intervals, moving from one spacer block to the next requires passing through the roof at \( 1 \) exactly \( r_1 \) consecutive times. Therefore, we can construct \( r_1 \) different sequences \( a_{in} \), each describing the height achieved after passing through the roof at \( 1 \) exactly \( n \leq r_1 \) times. In other words, for any of these \( r_1 \) sequences, \( a_{in}q\mod 1 \) describes the orbit of the \( n \)-th return map of an irrational rotation by \( q \) on the interval \( [0,q) \), which is known to be uniquely ergodic. This implies that the heights attained by the spacer blocks are described by the sequence \( a_{ir_1}q \mod 1 \), which equidistributes in \( [0,q) \).\\

Just as in the classical staircase case, if we enumerate the spacer blocks by counting the number of intervals they contain, we aim to describe, for each \( n \), a sequence \( n_i \) such that the \( n_k \)-th spacer block contains exactly \( n \) intervals.\\

For each \( n \), define the function
\[
g(n) = \prod_{i=2}^{n} r_i.
\]
Let \( n \) be any number, and consider \( m(n) \), the first step in the staircase configuration where a spacer block with \( n \) intervals appears. At step \( m(n) \), there will be \( m(n)-1 \) new values \( c_i \) corresponding to block sizes that were not present in the \( (m(n)-1) \)-th step. Ordering these from bottom to top, let \( F(n) \) denote the first occurrence of \( n \) in this list of blocks. If \( n \neq c_{m(n)-1} \), then  
\[
n_j = F(n) + j g(m(n)),
\]
and if \( n = c_{m(n)-1} \), then  
\[
n_j = F(n) + j g(m(n) + 1).
\]

In this setting, for every \( n \), the sequence \( a_{r_1n_j}q \mod 1 \) represents the orbit of \( g^{F(n)r_1}(0) \) under the transformation \( g^{r_1g(m(n)+1)} \) or \( g^{r_1g(m(n))} \), where \( g \) is the first return map of an irrational rotation by \( q \) on the subinterval \( [0,q) \). Consequently, the sequences \( a_{r_1n_j}q \mod 1 \) equidistribute in \( [0,q) \).\\

For the spacer part of the polygon, as observed in the previous paragraph, the sequences \( a_{r_1n_j}q \mod 1 \) describe the heights attained by spacer blocks containing \( n \) intervals. These sequences equidistribute in \( [0,q) \). It follows that if \( P_n \) denotes the probability of finding a spacer block with \( n \) intervals, then  
\[
\lim_{t\to\infty} \frac{\mathbf{card}(A_t^{s}(I^{(m)})\cap [a,b])}{\mathbf{card}(A_t^{s}(I^{(m)}))}=\sum_{n\in\mathbb{N}}\left(\frac{b-a}{q}\right)P_n=\frac{b-a}{q}.
\]

\end{proof}

As discussed in the previous section, it is not enough to consider $H$-equidistribution. 
\begin{lemma} \label{LemaHstrong}
    Each $I^{(m)}$ is $\overline{H}$-equidistributed.
\end{lemma}
\begin{proof}
    Observe that the staircase-like behavior described in \ref{Behaviour} and \ref{H2equidistribution} also holds for a flow built under any staircase transformation. Using the ideas from Lemma \ref{H2equidistribution}, the proof follows directly from Lemma \ref{lemma5.1}, Lemma \ref{H2equidistribution}, and Proposition \ref{propo5.4}.
\end{proof}
\begin{theorem} \label{generalcase}
    Let $T$ represent any mixing staircase transformation, and let $f$ be a $2$-wise constant function characterized by roof values $p$ and $q$ that are rationally independent. If all the spacers of the transformation $T$ are positioned beneath the roof with value $q$, then the flow constructed under $f$ respecting $T$ is  mixing.
\end{theorem}
\begin{proof}
    Just like in the proof of Theorem \ref{teo4.14}, we make use of Lemma \ref{lemma5.1} and Lemma \ref{LemaHstrong} to prove that there exists $m_0$ such that for any $m\geq m_0$, for any interval $I^{(m)}$ in the $m$-step configuration of the staircase and for any rectangle $R$ $$\lim_{t\to\infty}\lambda(I^{(m)}\cap \phi_{-t}(R))\geq \lambda(I^{(m)})\lambda(b(R))\lambda(h(R))=\lambda(I^{(m)})\mu(R).$$

    Following the same idea as in Theorem \ref{teo4.14} we check that the conditions for Lemma \ref{LemaUlci} are fulfilled, then $\phi_t$ is mixing.
\end{proof}

\section{$n$-fold mixing property} \label{nfold}

In the first half of last century, Rokhlin defined the notion of being $n$-fold mixing \cite{Ro}:
\begin{definition}

We say that a measurable dynamical system $(X,T,\mu)$ is $n$-fold mixing if for every $A_1,...,A_n$ measurable sets it is true that $$\lim_{m_1,...,m_{n-1}\to\infty}\mu(A_1\cap T^{m_1}(A_2)....\cap T^{m_1+m_2+...+m_{n-1}}(A_n))=\prod_{i=1}^n\mu(A_i).$$
\end{definition}

Note that $1$-fold mixing is just the notion of being strongly mixing. Rokhlin \cite{Ro} asked if every mixing transformation is also $n$-fold mixing for any $n$. This question remains open.\\

 Ryzhikov \cite{Ryz} provided a result which partially answers affirmatively this question for mixing rank one transformations and rank one flows.\\

In the preliminaries we provided a definition of rank-one transformations. There is an alternative definition that follows from the Rokhlin column theorem for aperiodic \footnote{Where the set of periodic points has zero measure.} transformation $T$.

\begin{definition}
    Let $(X, T,\mu)$ be a measurable dynamical system. We say $T$ is of rank-one if there exists a sequence of partitions of $X$ given by

    $$\eta_i=\{A_i, T(A_i),...,T^{h_i-1}(A_i),X\setminus{\bigcup_{j=0}^{h_i}T^j(A_i)}\}$$
    Such that the measure of the elements of those partitions goes to zero uniformly as $i\to\infty$.
\end{definition}

By Rokhlin column theorem, every aperiodic transformation $T$ has a representation like that, which is called a Rokhlin column, and $h_i$ is the height of such column.\\

We can observe that because of the construction of staircase transformation, they are rank-one transformations.\\

There are some generalizations of such transformations to flows, we are going to use the definition used by Ryzhikov \cite{Ryz}.

\begin{definition}

A \textit{rank-one flow} \( (X, \mathcal{B}, \mu, \phi_t) \) is a measure-preserving flow on a probability space such that there exists a sequence of measurable sets \( A_n \) (the \textit{tower bases}), a  decreasing sequence of times \( t_n \) and an increasing sequence of \textit{heights} \( h_n \) such that:
\begin{enumerate}
    \item \( t_n \to 0 \) as \( n \to \infty \).\\
    \item The total height satisfies \( h_n t_n \to \infty \), meaning that the towers approximate arbitrarily large segments of the flow.\\
    
    \item The sets \( \{ \phi_{k t_n}(A_n) \}_{k=0}^{h_n-1} \) are essentially disjoint, meaning for every $k_1\neq k_2$
    \[
    \mu \left( \phi_{k_1t_n}A_n \cap \phi_{k_2 t_n}(A_n) \right) = 0, \quad \text{for all } k_1,k_2 = 0, \dots, h_n-1.
    \]\\
    \item The space \( X \) is well approximated in measure by the union of these towers, i.e.,
    \[
    \mu \left( X \setminus \bigcup_{k=0}^{h_n -1} \phi_{kt_n}(A_n) \right) \to 0 \quad \text{as } n \to \infty.
    \]\\
    
\end{enumerate}
\end{definition}

If $(X,\phi_t,\mu)$ is the flow considered in the previous two sections, since it is a suspension over a rank one transformation and the roof function is piecewise constant it is immediate that:

\begin{lemma}
   Let $T$ denote any  staircase transformation, and let $f$ be a roof function that is $2$-wise constant, characterized by two height values that are rationally independent. Consider $\phi_t$ as the flow constructed under $f$ for $T$. Then $\phi_t$ is a rank one flow. 
\end{lemma}

In section 4 in \cite{Ryz}, Ryzhikov proved that 

\begin{theorem} \label{ryz}
    (Ryzhikov 1993 \cite{Ryz}).\\
    If $\phi_t$ is a mixing rank-one flow, then $\phi_t$ is $n$-fold mixing for any $n$.
\end{theorem}

In sight of Theorem \ref{generalcase} and Theorem \ref{ryz} it follows that:

\begin{corollary}
     Let $T$ denote any mixing staircase transformation, and let $f$ be a roof function that is $2$-wise constant, characterized by two height values that are rationally independent. Consider $\phi_t$ as the flow constructed under $f$ for $T$. Then, the flow $\phi_t$ is $n$-fold mixing for any $n$.
\end{corollary}

\section{Conclusion}
Exploring the case of general rank-one mixing transformations would be a natural next step. Our results for mixing staircases suggest that, under additional suitable conditions on the rank-one construction, mixing in the associated suspension flows could potentially be established. Notably, the first part of our proof relies only on the mixing properties of $T$, which are well understood for rank-one transformations. However, the second part crucially depends on the staircase-like structure to obtain sequences of heights with controlled properties. Extending this approach to general mixing rank-one transformations would likely require further refinement and deeper investigation.
\\

{\bf Acknowledgement.} I sincerely thank my advisor, Anja Randecker, for her encouragement and meticulous reading of this work. I am also grateful to Rodrigo Treviño for bringing this problem to my attention.

\bibliographystyle{alpha}
\bibliography{bibliography}

\end{document}